\documentclass[12pt]{amsart}
\usepackage[matrix,tips,arrow,curve]{xy}
\usepackage[unicode]{hyperref}
\usepackage{graphicx}
\usepackage{MnSymbol}
\usepackage{footnote}

\SelectTips{lu}{12}

\theoremstyle{plain}
\newtheorem{thm}[equation]{Theorem}
\newtheorem{lem}[equation]{Lemma}
\newtheorem{prop}[equation]{Proposition}
\theoremstyle{definition}
\newtheorem{defn}[equation]{Definition}
\newtheorem{example}[equation]{Example}
\newtheorem{rem}[equation]{Remark}

\newcommand{\D}{\mathcal{D}}

\newcommand{\id}{\operatorname{id}}
\newcommand{\inv}{{-1}}
\newcommand{\set}[1]{\{#1\}} 
\newcommand{\setm}[2]{\{\,#1\mid#2\,\}}
\newcommand{\Set}{\mathbf{Set}}
\newcommand{\mcd}{\mathcal{D}}
\newcommand{\N}{\mathbb{N}}
\newcommand{\ci}{\subset}
\newcommand{\Hom}{\operatorname{Hom}}
\newcommand{\End}{\mathcal{E}nd}
\newcommand{\obj}{\operatorname{Ob}}
\newcommand{\ob}{\obj}
\newcommand{\Prop}{\ensuremath{\mathbf{Prop}}}
\newcommand{\Perm}{\ensuremath{\mathbf{Perm}}}
\newcommand{\Cat}{\ensuremath{\mathbf{Cat}}}
\newcommand{\Operad}{\ensuremath{\mathbf{Operad}}}
\newcommand{\Mega}{\ensuremath{\mathbf{Mega}}}

\renewcommand{\C}{\mathcal{C}}			
\newcommand{\T}{\mathcal{T}}				
\newcommand{\Q}{\T}						
\newcommand{\SSS}{\mathcal{S}}			
\newcommand{\R}{\mathcal{R}}			
\newcommand{\inthom}{\mathcal{H}om}		
\newcommand{\X}{\mathcal{X}}
\newcommand{\Y}{\mathcal{Y}}

\newcommand{\bilinue}{\operatorname{Bilin}}
\newcommand{\bilin}{\mathcal{B}ilin}
\newcommand{\col}{\operatorname{Col}}
\newcommand{\someoperad}{\mathcal{O}}
\newcommand{\lis}[1]{\left< #1 \right>}   
\newcommand{\M}{\mathbb{M}}

\newcommand{\blis}{\lis}	

\newcommand{\tin}{\operatorname{in}}
\newcommand{\tout}{\operatorname{out}}

\def\Jac#1#2#3{{%Picture saved by xtexcad 2.4
\unitlength=.4pt
\begin{picture}(48.00,30.00)(-4,0.00)
\bezier{34}(20.00,20.00)(30.00,10.00)(40.00,0.00)
\bezier{34}(20.00,20.00)(10.00,10.00)(0.00,0.00)
\bezier{20}(10.00,10.00)(15.00,5.00)(20.00,0.00)
\put(20.00,30.00){\line(0,-1){10.00}}
\put(0,-5){\makebox(0,0)[t]{\scriptsize $#1$}}
\put(20,-5){\makebox(0,0)[t]{\scriptsize $#2$}}
\put(40,-5){\makebox(0,0)[t]{\scriptsize $#3$}}
\end{picture}}}

\def\coJac#1#2#3{{% Picture saved by xtexcad 2.4
\unitlength=.4pt
\begin{picture}(48.00,30.00)(-4,-30.00)
\bezier{30}(20.00,-20.00)(30.00,-10.00)(40.00,0.00)
\bezier{30}(20.00,-20.00)(10.00,-10.00)(0.00,0.00)
\bezier{20}(10.00,-10.00)(15.00,-5.00)(20.00,0.00)
\put(20.00,-30.00){\line(0,1){10.00}}
\put(0,5){\makebox(0,0)[b]{\scriptsize $#1$}}
\put(20,5){\makebox(0,0)[b]{\scriptsize $#2$}}
\put(40,5){\makebox(0,0)[b]{\scriptsize $#3$}}
\end{picture}}}

\def\dvojiteypsilonvetsi#1#2#3#4{{
\unitlength=.46pt
\begin{picture}(24.00,30.00)(0.00,3.00)
\put(10.00,20.00){\line(0,-1){10.00}}
\bezier{20}(10.00,10.00)(15,5)(20.00,0.00)
\bezier{20}(10.00,10.00)(5,5)(0.00,0.00)
\bezier{20}(10.00,20.00)(15,25)(20.00,30.00)
\bezier{20}(0.00,30.00)(5,25)(10.00,20.00)
\put(0,-5){\makebox(0,0)[t]{\scriptsize $#1$}}
\put(20,-5){\makebox(0,0)[t]{\scriptsize $#2$}}
\put(0,35){\makebox(0,0)[b]{\scriptsize $#3$}}
\put(20,35){\makebox(0,0)[b]{\scriptsize $#4$}}
\end{picture}}}

\def\pravaplastevvetsi#1#2#3#4{
\unitlength=.45pt
\begin{picture}(38,30)(-4,4)
\put(0.00,0.00){\line(0,1){10}}
\put(20.00,0.00){\line(0,1){10}}
\put(10.00,20.00){\line(0,1){10}}
\put(30.00,20.00){\line(0,1){10}}
\put(0,10){\bezier{20}(0.00,0.00)(5.00,5.00)(10.00,10.00)}
\put(20,10){\bezier{20}(0.00,0.00)(5.00,5.00)(10.00,10.00)}
\put(20,10){\bezier{20}(0.00,0.00)(-5.00,5.00)(-10.00,10.00)}
\put(0,-5){\makebox(0,0)[t]{\scriptsize $#1$}}
\put(20,-5){\makebox(0,0)[t]{\scriptsize $#2$}}
\put(10,35){\makebox(0,0)[b]{\scriptsize $#3$}}
\put(30,35){\makebox(0,0)[b]{\scriptsize $#4$}}
\end{picture}
}

\def\levaplastevvetsi#1#2#3#4{
\unitlength=.45pt
\begin{picture}(38,30)(-34,4)
\put(0.00,0.00){\line(0,1){10}}
\put(-20.00,0.00){\line(0,1){10}}
\put(-10.00,20.00){\line(0,1){10}}
\put(-30.00,20.00){\line(0,1){10}}
\put(0,10){\bezier{20}(0.00,0.00)(-5.00,5.00)(-10.00,10.00)}
\put(-20,10){\bezier{20}(0.00,0.00)(-5.00,5.00)(-10.00,10.00)}
\put(-20,10){\bezier{20}(0.00,0.00)(5.00,5.00)(10.00,10.00)}
\put(0,-5){\makebox(0,0)[t]{\scriptsize $#2$}}
\put(-20,-5){\makebox(0,0)[t]{\scriptsize $#1$}}
\put(-10,35){\makebox(0,0)[b]{\scriptsize $#4$}}
\put(-30,35){\makebox(0,0)[b]{\scriptsize $#3$}}
\end{picture}
}

\title{On the category of props}
\author[P. Hackney]{Philip Hackney}
\address{Department of Mathematics, University of California-Riverside}
\email{hackney@math.ucr.edu}

\author[M. Robertson]{Marcy Robertson}

\address{Department of Mathematics, University of Western Ontario, Canada}
\email{mrober97@uwo.ca}


\begin{document}
\begin{abstract}The category of (colored) props is an enhancement of the
category of colored operads, and thus of the category of small
categories. The titular category has nice formal properties: it is
bicomplete and is a symmetric monoidal category, with monoidal product
closely related to the Boardman-Vogt tensor product of operads. Tools
developed in this article, which is the first part of a larger work, include a
generalized version of multilinearity of functors, a free prop
construction defined on certain ``generalized'' graphs, and the
relationship between the category of props and the categories of
permutative categories and of operads.\end{abstract} 
\maketitle

\section{Introduction}

This paper lays the foundation for a multistage project developing the
notion of  ``higher prop.'' Here we establish the formal properties
necessary to do homotopy theory: the category of props (enriched over a
suitable symmetric monoidal category) is complete, cocomplete
(Theorem~\ref{T:cocomplete}), and closed symmetric monoidal
(Theorem~\ref{T:csm}). These are the basic properties required for the
sequel \cite{hackneyrobertson2}, where we construct a cofibrantly
generated model structure on the category of props enriched in
simplicial sets. Later papers will develop combinatorial models for
up-to-homotopy props, following the dendroidal approach \cite{cm-ho, mw,
mw2} as well as comparisons between these models (as in \cite{cm-ds,
cm-simpop} in the dendroidal setting). For motivation for the project as
a whole, see the introduction to \cite{hackneyrobertson2}.

Operads are a tool used to model (co)algebraic structures, i.e
associative, associative and commutative, co-associative co-commutative,
Lie, Poisson, etc. They were first introduced in the 70's in algebraic
topology \cite{bv, geometry}, experienced a renaissance in the 90's
\cite{ginzburgkapranov}, and now are ubiquitous throughout various areas
of mathematics (see~\cite{mss} for a survey). Operads can model
structures which have operations with multiple inputs and a single
output. A basic example is the operad $\mathbf{Ass}$, whose algebras are
the monoids. Coalgebras over $\mathbf{Ass}$ are precisely comonoids, so
we see that operads can also be used to model structures with
co\"operations \cite[3.71]{mss}. 

Operads, however, cannot model all algebraic structures of interest. For
example, it is well known that there is no operad which models groups. A
shadow of this fact occurs when we work over $k$-modules: algebras over
$\mathbf{Ass}$ are $k$-algebras, coalgebras over $\mathbf{Ass}$ are
$k$-coalgebras, but there is no operad which models Hopf algebras, which
possess both multiplication and comultiplication. In order to study
families of algebras of this type, one must pass to the strictly richer
category of props, which was introduced by MacLane \cite{catalg} long
before the invention of operads. Simply put, a prop can control mixed
algebraic and coalgebraic structures, like Hopf algebras. Another
important example are the various cobordism categories, which may be
minimally described by a prop. Thus certain varieties of field theories
are algebras over a prop whose morphisms are cobordisms. 

Both of these examples are monochrome props, so why might one consider
props with more general color sets? Let us justify this with an example,
a subcategory, and an area of application. The example is the existence
of a $2$-colored prop, whose algebras consist of two Hopf algebras
together with a morphism from the first to the second. More generally,
given a prop T there is a prop S so that algebras over S are maps of
T-algebras, although this comes at the cost of a doubling of colors.
Secondly, we allow general color sets because we may then consider props
as generalized categories. The category of small categories embeds in
the category of props (see section \ref{S:adjcat}). Finally, props are
an ideal way to study certain problems in computer science. We could
consider a prop whose morphisms are functions with multiple inputs and
outputs, e.g. a function which takes a full customer record and returns
the customer's phone number and name. The colors of this prop are the
various data types of the language ({\tt int}, {\tt float}, {\tt
string}, etc.) and user-defined types. Propic composition can describe
piping inputs of some functions into outputs of others, and can also
describe parallel execution of functions. We will discuss 
several other examples in detail in section~\ref{S:examples}.

\subsection{Acknowledgments} We would like to thank the referee for pointing out an oversight in a previous version of our free prop construction.

\subsection{Definition of prop} Intuitively, a prop is a generalization
of a category. We still have sets of objects, but arrows $x\rightarrow
y$ are replaced by multilinear operations which may have $n$-inputs and
$m$-outputs, i.e. maps are multilinear maps $x_{1}\otimes \dots \otimes
x_{n}\rightarrow y_{1}\otimes \dots \otimes y_m$. In the monochrome
case, a prop may be defined as a symmetric monoidal category freely
generated by a single object. The intuition is that this is essentially
a generalization of Lawvere theories that work in non-Cartesian
contexts. 
More explicitly, a \emph{prop} $\T$ consists of the following
data:

\begin{itemize}
\item A \emph{set} of colors\footnote{aka `objects'} $C = \col(\T)$,
\item for every (ordered) list of colors $a_1, \dots, a_n, b_1, \dots, b_k \in C$ (where $n,m\geq 0$), a set of operations 
	\[
	\T(a_1, \dots, a_n; b_1, \dots, b_m) = \T(\lis{a_i}_{i=1}^n; \lis{b_k}_{k=1}^m),\] \label{enumeratedhomsets}
\item a specified element $\id_c \in \T(c; c)$ for each $c\in C$,
\label{enumeratedid}
\item an associative \emph{vertical composition} \begin{align*}
\T(\lis{a_i}_{i=1}^n; \lis{b_k}_{k=1}^m) \times \T(\lis{c_j}_{j=1}^p;
\lis{a_i}_{i=1}^n) & \to \T(\lis{c_j}_{j=1}^p ; \lis{b_k}_{k=1}^m) \\
(f,g) &\mapsto f\circ_v g,\end{align*}
\item an associative \emph{horizontal composition} \begin{align*} \T(
\lis{a_i}_{i=1}^n; \lis{b_k}_{k=1}^m)\times \T(\lis{a_i}_{i=n+1}^{n+p};
\lis{b_k}_{k=m+1}^{m+q}) &\to \T( \lis{a_i}_{i=1}^{n+p}
;\lis{b_k}_{k=1}^{m+q} )\\ (f,g) &\mapsto f\circ_h g, \end{align*}
\label{enumeratedhoriz}
\item a map $\sigma^*: \T(\lis{a_i}_{i=1}^n; \lis{b_k}_{k=1}^m) \to
\T(\lis{a_{\sigma(i)}}_{i=1}^n; \lis{b_k}_{k=1}^m) $ for every element
$\sigma\in \Sigma_n$, and
\item a map $\tau_*: \T(\lis{a_i}_{i=1}^n; \lis{b_k}_{k=1}^m)
\to\T(\lis{a_i}_{i=1}^n; \lis{b_{\tau^\inv(k)}}_{k=1}^m)$ for every
$\tau \in \Sigma_m$. \end{itemize}
We typically utilize the notation $\lis{a_i}_{i=1}^n$ or $\blis{a_1, \dots, a_n}$ for a (possibly empty) list of elements $a_1, \dots, a_n$, 
omitting the brackets where appropriate (e.g. when $n=1$).
We will frequently denote elements of the set $\T(\lis{a_i}_{i=1}^n; \lis{b_k}_{k=1}^m)$ by
\[ f:\lis{a_i}_{i=1}^n\rightarrow \lis{b_k}_{k=1}^m.\]
The above data are required to satisfy the following axioms:
\begin{itemize}
\item The elements $\id_c$ are identities for the \emph{vertical} composition, i.e.
\begin{equation} \label{E:identities}
\begin{aligned}
f\circ_v (\id_{a_1} \circ_h \dots \circ_h \id_{a_n}) &=f \\
(\id_{b_1} \circ_h \dots \circ_h \id_{b_m}) \circ_v f &=f.
\end{aligned} \end{equation}
\item The horizontal and vertical compositions satisfy an interchange
rule \begin{equation} \label{E:hvcompat} (f\circ_v g) \circ_h (f'
\circ_v g') = (f\circ_h f') \circ_v (g\circ_h g') \end{equation}
\emph{whenever the vertical compositions on the left are well-defined}.
\item  The vertical composition is compatible with the symmetric group
actions in the sense that
\begin{equation} \label{E:vertcompat}
\begin{aligned} f\circ_v (\sigma_* g) &= (\sigma^* f) \circ_v g \\
\sigma^*(f\circ_v g) &= f \circ_v (\sigma^* g) \\
\tau_*(f\circ_v g) &= (\tau_* f)\circ_v g, 
\end{aligned}
\end{equation}
where $\tau$ and $\sigma$ are permutations on the appropriate number of letters.
\item Suppose that $f$ has $n$ inputs and $m$ outputs and $g$ has $p$
inputs and $q$ outputs. If $\sigma \in \Sigma_n$, $\bar\sigma \in
\Sigma_p$, $\tau\in \Sigma_m$, and $\bar\tau \in \Sigma_q$ and we write 
$\sigma \times \bar\sigma \in \Sigma_n \times \Sigma_p \hookrightarrow
\Sigma_{n+p}$ then the horizontal composition satisfies
\begin{equation} \label{E:horizcompat}
\begin{aligned}
(\sigma^* f) \circ_h (\bar\sigma^* g) &= (\sigma \times \bar{\sigma})^*(f\circ_h g) \\
(\tau_* f) \circ_h (\bar\tau_* g) &= (\tau \times \bar{\tau})_*(f\circ_h g). \\
\end{aligned}
\end{equation}
Furthermore, if $\sigma_{xy} \in \Sigma_{x+y}$ is the permutation whose
restrictions are increasing bijections
\begin{align*} \sigma_{xy}: [1,y] &\overset\cong\longrightarrow [x+1, x+y] \\
\sigma_{xy}: [1+y, x+y] &\overset\cong\longrightarrow [1,x]
\end{align*}
then 
\begin{equation}
(\sigma_{p,n})^*(\sigma_{m,q})_* (f\circ_h g) = g\circ_h f. \label{E:horizswap}
\end{equation}
\item The maps $\sigma^*$ and $\tau_*$ satisfy the interchange rule
$\sigma^* \tau_* = \tau_* \sigma^*$ and are \emph{actions}: \[ \sigma^*
\bar\sigma^* = (\bar\sigma \sigma)^* \qquad \qquad \tau_* \bar\tau_* =
(\tau \bar\tau)_*. \]
\end{itemize}

\begin{rem}  There are several variations on the term ``prop'' in the
literature.  Boardman and Vogt use the name `colored PROP' for a prop
which is completely determined by operations with $n$-inputs and only
one output \cite[Definition 2.44]{bv}. With this definition colored
PROPs are the same thing as colored operads or multicategories (see 
\cite[2.3.1]{leinster}). The definition we have given here is in line
with the original due to MacLane, see, for instance, \cite{fmy, fresse,
jy, maclane, operadsandprops}.\end{rem}

\begin{defn} A homomorphism of props $f:\R\rightarrow\T$ consists of a  map $\col(\R)\rightarrow\col(\T)$ and for each input-output profile $a_1,\dots ,a_n;b_1,\dots ,b_m$ in $\col(\R)$ a map $\R(\lis{a_{i}}_{i=1}^{n};\lis{b_j}_{j=1}^m)\longrightarrow\T(\lis{fa_{i}}_{i=1}^{n};\lis{fb_j}_{j=1}^m)$ which commutes with all composition, identity, and symmetry operations. The category of props and prop homomorphisms is denoted $\Prop$. \end{defn} 

It is straightforward to generalize the definitions from this section to a \emph{prop
enriched in a symmetric monoidal category} $(\mathcal{E}, \boxtimes,
I)$. In the data, one replaces the sets of operations
$\T(\lis{a_i}_{i=1}^n; \lis{b_k}_{k=1}^m)$ with objects of
$\mathcal{E}$, the specified elements $\id_c$ by maps $I\to \T(c;c)$,
and all products $\times$ by $\boxtimes$. In the axioms, all equalities
on elements should be expressed instead by requiring that the relevant
diagrams commute. 

\subsection{Examples of props}\label{S:examples}

\begin{example} The Segal prop (see \cite{Seg}) is a prop of infinite dimensional
complex orbifolds. The space of morphisms is defined as the moduli space
$P_{m,n}$ of complex Riemann surfaces bounding $m+ n$ labeled
nonoverlapping holomorphic holes. The surfaces should be understood as
compact smooth complex curves, not necessarily connected, along with $m
+ n$ biholomorphic maps of the closed unit disk to the surface. The
precise nonoverlapping condition is that the closed disks in the inputs
(outputs) do not intersect pairwise and an input disk may intersect an
output disk only along the boundary. This technicality brings in the
symmetric group morphisms, including the identity, to the prop, but does
not create singular Riemann surfaces by composition. The moduli space
means that we consider isomorphism classes of such objects. The
composition of morphisms in this prop is given by sewing the Riemann
surfaces along the boundaries, using the equation $zw = 1$ in the
holomorphic parameters coming from the standard one on the unit disk.
The tensor product of morphisms is the disjoint union. This prop plays a
crucial role in conformal field theory.  

\end{example} 

\begin{example} Suppose that $C$ is a set and we have a family
$\mathbf{X} = \set{X_c}_{c\in C}$ of objects in some symmetric monoidal
category $(\mathcal{E}, \boxtimes, I)$. This data determines an
prop\footnote{We can think of this as an $\mathcal{E}$-prop if
$\mathcal{E}$ is \emph{closed} symmetric monoidal.} $\End_{\mathbf{X}}$
with color set $C$, called the \emph{endomorphism prop of $\mathbf{X}$}.
It is defined by \[ \End_{\mathbf{X}} (c_1, \dots, c_n; d_1, \dots, d_m)
= \mathcal{E}(X_{c_1} \boxtimes \dots \boxtimes X_{c_n}, X_{d_1}
\boxtimes \dots \boxtimes X_{d_m}),  \] together with the
$\Sigma$-actions and compositions coming from the monoidal structure. As
in the case of operad algebras, if $\T$ is a prop, then a
\emph{$\T$-algebra} is a prop map $\T \to \End_{\mathbf{X}}$.
\end{example} 

\begin{example}~\cite[2.1.4]{stringtop}A more geometric example, due to Sullivan, is the Lie
bialgebra prop. A Lie \emph{bialgebra} is a Lie algebra $\frak{g}$ with
the structure of a Lie coalgebra given by a one-cocycle $\delta
:\frak{g}\rightarrow\frak{g}\wedge\frak{ g}$ on $\frak{g}$ with values
in the $\frak{g}$-module $\frak{g}\wedge\frak{g}$, i.e., the linear map
$\delta(g)$ satisfies the cocycle condition: $$\delta([g_1,g_2]) =
g_1\delta(g_2) \minus g_2\delta(g_1)$$ for all $g_1, g_2\in\frak{ g}$. Lie
bialgebras are so-called quasi-classical limits of quantum groups (more
precisely, quantum universal enveloping algebras) and they play a key
role in deformation theory (see, for example \cite{fmy,pirashvili}). One
can construct a monochrome prop  which has $(n,m)$-ary operations
defined as quotient spaces of vector spaces spanned by graphs of a
certain type. Explicitly, when $m$ or $n$ is $0$, we define $L(n,m) :=
0$. For $m, n \ge 1$, the space $L(n,m)$ may be defined as follows. 

Consider the vector space spanned freely by the (isomorphism classes of)
directed oriented trivalent graphs $\Gamma$ with $n+m$ legs labeled as
inputs $1, \dots ,n$ and outputs $1, \dots , m$. The graphs need not be
connected, but must be finite. A leg is either an edge whose one end is
free, that is, not a vertex, while the other end is a vertex, or a
half-edge of an edge with two free ends. The adjective directed refers
to the choice of directions on each edge, so that the legs are directed
from the inputs and toward the outputs and the directions define a
partial order on the set of vertices. Trivalent here means that all
vertices must have one incoming and two outgoing edges or two incoming
and one outgoing edges. Graphs with no vertices, i.e., disjoint unions
of edges each of which connects an input with an output, are allowed. An
orientation on a graph means the choice of an ordering on the set of
edges, up to the sign of a permutation. We define $L(n,m)$ to be the
quotient of this space of graphs by relations generated by 

\vglue 3pt
\[
\Jac 123 + \Jac 231 + \Jac 312 \hskip 1mm , \
\coJac 123 +  \coJac 231 + \coJac 312 \hskip 1mm   \mbox { and }
\hskip 2mm 
\dvojiteypsilonvetsi1212 - 
\levaplastevvetsi1212  - 
\pravaplastevvetsi1212 +
\levaplastevvetsi1221 + \pravaplastevvetsi1221 \hskip 1mm, 
\]
with labels indicating, in the obvious way, the corresponding
permutations of the inputs and outputs.\end{example} 

\subsection{Relationship with colored operads}
Recall that a colored \emph{operad}\footnote{which is variously called 
`symmetric multicategory' or simply `operad'} $\someoperad$ is a
structure with a color set $\col \someoperad$ and hom sets
$\someoperad(c_1, \dots, c_n; c)$ for each list of colors $c_1, \dots,
c_n,c$, together with appropriate composition operations. The precise
definition is a bit more involved than that of prop (see
\cite{bmresolution}) since the operadic composition mixes together
horizontal and vertical propic compositions. However, we can regard
colored operads as a special type of prop, namely those which are
completely determined by the `one-output part'. Let $\Operad$ be the
category of colored operads, which we shall now refer to simply as
\emph{operads} (see \cite{moerdijklecture}). There is a forgetful
functor \[ U: \Prop \to \Operad \] which takes $\T$ to an operad $U(\T)$
with $\col \T = \col U(\T)$. The morphism sets are defined by \[ U(\T)
(a_1, \dots, a_n; b) = \T(a_1, \dots, a_n; b). \]  Operadic composition
\[ \gamma: U(\T)(\lis{a_i}_{i=1}^n; b) \times \prod_{i=1}^n U(\T)
(\lis{c_{i,j}}_{j=1}^{p_i}; a_i) \to U(\T)\left(\lis{
\lis{c_{i,j}}_{j=1}^{p_i}}_{i=1}^n; b\right) \] is then given by  \[
\gamma(g, \lis{f_i}_{i=1}^n) = g\circ_v (f_1 \circ_h \dots \circ_h f_n).
\]

\begin{prop}\label{P:operadpropadjunction}
The forgetful functor $U: \Prop \longrightarrow\Operad$ has a left
adjoint $F: \Operad \longrightarrow \Prop$ with $\col F(\someoperad) =
\col \someoperad$. An element of $F\someoperad(a_1, \dots, a_n ; b_1,
\dots b_m)$ is given by $(\theta, \lis{f_j}_{j=1}^m)$ where
$\theta\colon \set{1, \dots, n} \to \set{1, \dots, m}$ is a function and
$ f_j: \lis{ a_{i} }_{\theta(i)=j} \to b_j $ is in $\someoperad$.
\end{prop}

\begin{proof} To show that we have a prop we need to first define the
vertical and horizonal composition relations.  Vertical composition of 
$(\theta, \lis{f_j})$ and $(\phi, \lis{g_k})$ is defined to be
\begin{gather*} \psi = (\phi \circ \theta , \lis{h_k}) \\h_k =
\gamma(g_k, \left< f_j \right>_{\phi(j) = k})\end{gather*} where
$\gamma$ is the operadic composition. Consider $(\theta_1,
\lis{f_{1j}}), (\theta_2, \lis{f_{2j}})$, where \begin{align*} \theta_1:
\set{1, \dots, n_1} &\to \set{1, \dots, m_1} \\ \theta_2: \set{1, \dots,
n_2} &\to \set{1, \dots, m_2} \end{align*} and $f_{\ell j}: \lis{a_{\ell
i}}_{\theta_\ell(i) = j} \to b_{\ell j}$. The horizontal composition of
these is $(\theta, f_k)$, where \begin{align*} \theta: \set{ 1,\dots,
n_1 + n_2} &\to \set{1, \dots, m_1 + m_2} \\ i &\mapsto \begin{cases}
\theta_1(i) & i\leq n_1 \\ \theta_2(i-n_1) + m_1 & i > n_1
\end{cases}\end{align*} and \[ f_k = \begin{cases} f_{1k} & 1\leq k \leq
m_1 \\ f_{2(k-m_1)} & m_1+1 \leq k \leq m_1 + m_2.\end{cases} \]

If $\tau$ is an element of $\Sigma_m$,  we define the left action \[
\tau_*: F(\someoperad)(\lis{a_i}_{i=1}^n; \lis{b_k}_{k=1}^m) \to
F(\someoperad)(\lis{a_i}_{i=1}^n; \lis{b_{\tau^\inv(k)}}_{k=1}^m) \] by
$(\theta, \lis{f_j}_{j=1}^m) \mapsto (\tau \circ \theta,
\lis{f_{\tau^\inv (j)}}_{j=1}^m)$. If $\sigma$ is an element of
$\Sigma_n$ then we define the right action \[ \sigma^*:
\T(\lis{a_i}_{i=1}^n; \lis{b_k}_{k=1}^m) \to
\T(\lis{a_{\sigma(i)}}_{i=1}^n; \lis{b_k}_{k=1}^m) \]  by $(\theta,
\lis{f_j}_{j=1}^m) \mapsto (\theta \circ \sigma,
\lis{\gamma_j^*f_j}_{j=1}^m)$. Here, $\gamma_j$ is the composition \[
\gamma_j: \theta^\inv(j) \to \sigma^\inv
\theta^\inv(j)\overset\sigma\to\theta^\inv(j)\] where the first map is
the order preserving bijection. It is now left as an exercise to verify
that the axioms of a prop are satisfied. 

Let $*: \set{1, \dots, n} \to \set{1}$ denote the unique map.  We see
that a map of operads $q: \someoperad \to U(\T)$ uniquely determines a
map of props $q': F(\someoperad) \to \T$ so that $(U(q'))(*, f) =  q(f)$.
Thus, $F$ is left adjoint to $U$. 
\end{proof}

Since there is only one map $\theta: \set{1, \dots, n} \to \set{1}$, we have
\begin{prop}\label{P:UFequalsID}
If $\someoperad$ is an operad then 
\[ F(\someoperad) (a_1, \dots, a_n; b) \cong \someoperad(a_1, \dots, a_n; b). \]
Consequently, $UF \cong \id_{\Operad}$. \qed
\end{prop}

\subsection{Relationship with categories}\label{S:adjcat} 
Informally, we can say that inside every operad lies a category which
makes up the linear part (i.e. the operations with one input and one
output) of that operad. In fact, we have an ``enrichment'' of the
category $\Prop$ over $\Cat$.\footnote{Coming from the fact that
$\Prop$ is enriched over itself; cf.~\ref{S:selfenriched}.} We can
assign to each operad $\someoperad$ a genuine category
$U_0(\someoperad)$ whose object set is the color set of $\someoperad$ and
has morphisms given by $U_0(\someoperad)(a,b):=\someoperad(a;b)$ for any
two colors $a,b$ in $\someoperad.$ Composition and identity operations
are induced by those of $\someoperad$. This relationship with category
theory is useful in making sense of ideas which do not have obvious
meaning in the setting of operads or props. 

The functor $U_0$ admits a left adjoint, denoted by $F_0$, which takes a
category $\C$ to an operad $F_0(\C)$ with $\obj(F_0\C):=\obj(\C)$. The
linear operations are just the composition maps of $\C$, i.e.
$F_0(\C)(a; b):=\C(a, b),$ and the higher operations are all trivial,
i.e. $F_0(\C)(a_1,...,a_n; b) =\varnothing$ for $n\neq 1$. Composition
and units are induced from $\C$ in the obvious way, and it is an easy
exercise to check the necessary axioms of an operad are satisfied.

\subsection{Graphs and megagraphs} 

We now fix our notion of (directed) graph, which is essentially the same
as that in \cite[A.1]{fresse}. The graphs in this paper have a finite
set of vertices $V$, a finite set of edges $E$, and functions
\begin{align*}
s: E &\to V_+ = V \sqcup \set{*} \\
t: E &\to V_+
\end{align*} 
which take an edge $e$ to its tail $s(e)$ and its head $t(e)$. Notice
that we allow for either of these to be trivial, i.e.  we allow
half-edges and edges that are incident to no vertices. A \emph{cycle} is
a list of edges $e_1, \dots, e_n$ such that $t(e_i) = s(e_{i+1})\in V$
and $t(e_n)=s(e_1)\in V$. We will want to work with graphs which do not
have cycles; in particular, we have no loops (cycles with $n=1$). We
will denote all the data of a graph by $G:=(E,V,s,t)$ and will write \[
\tin(v) = t^\inv(v) \qquad \qquad \tout(v)=s^\inv(v) \] for the sets of
input and output edges of a vertex.

A morphism of graphs $f: G \to G'$ consists of functions $f_E: E \to E'$
and $f_V: V_+ \to V'_+$ with $f_V(*) = *$, $f_V(v) \neq *$, $s f_E (e) =
f_V s(e)$, and $t f_E(e) = f_V t(e)$. The first two conditions ensure
that $f_E$ preserves the \emph{type} of edges, i.e. edges go to edges,
half-edges go to half-edges pointing in the same direction, and
non-incident edges go to non-incident edges. 

To define the underlying ``graphs'' of the category $\Prop$,  consider
the free monoid monad acting on a set $S$
\begin{align*}\label{freemonoidmonad}
\M: \Set &\to \Set \\
S & \mapsto \coprod_{k \geq 0} S^{\times k}. 
\end{align*}
Elements of $\M S$ are just (finite) ordered lists of elements of $S$.
There are right and left actions of the symmetric groups on the
components of $\M S$. More compactly we could say that there are both
right and left actions of the symmetric \emph{groupoid} $\Sigma =
\coprod_{n\geq 0} \Sigma_n$ on $\M S$. A $\Sigma$-bimodule is a set with
compatible left and right $\Sigma$-actions.

We now describe an extension of the notion of graph, namely one in which
edges are permitted to have multiple inputs and outputs.\footnote{Unlike
multigraphs, it makes sense to consider undirected megagraphs.} See
Figure~\ref{F:samplemega}.

\begin{figure} \includegraphics[height=2cm]{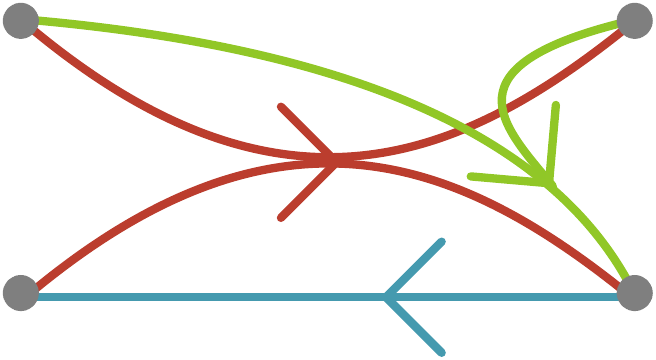} \caption{An
example megagraph with four vertices and three megaedges, each having a
different color.}\label{F:samplemega} \end{figure}

\begin{defn}A \emph{megagraph} $\X$ consists of a set of objects $X_0$,
a set of arrows $X_1$, two functions $s: X_1 \to \M X_0$ and $t: X_1 \to
\M X_0$, which we will write as the span \[ \M X_0 \overset{s}\leftarrow
X_1 \overset{t}\rightarrow \M X_0. \] Furthermore, $X_1$ should possess
both right and left $\Sigma$ actions.
These actions should have an interchange property $\tau \cdot (x\cdot
\sigma) = (\tau \cdot x) \cdot \sigma$ and should be compatible with
those on $\M X_0$, so $t( \tau \cdot x) = \tau \cdot t(x)$ and  $s(  x
\cdot \sigma) = s(x) \cdot \sigma$. \end{defn}

A map of megagraphs $f: \X \to \Y$ is determined by maps $f_0: X_0 \to
Y_0$ and $f_1: X_1 \to Y_1$ so that the diagram 
\[ \xymatrix{
\M X_0 \ar@{->}[d]^{\M f_0} & X_1 \ar@{->}[r]_-t \ar@{->}[l]^-s \ar@{->}[d]^{f_1}& \M X_0 \ar@{->}[d]^{\M f_0} \\
\M Y_0 & Y_1 \ar@{->}[r]_-t \ar@{->}[l]^-s  & \M Y_0
}\]
commutes. The collection of megagraphs determines a category which we
call $\Mega$. 

Notice that a megagraph $\X$ would be called a $X_0$-colored
$\Sigma$-bimodule in \cite{fmy}. We deal so frequently with color change
that the current viewpoint seems appropriate. We also would like to
point out that every megagraph has an underlying \emph{directed
hypergraph} (see \cite{hypergraph,torres}) obtained by forgetting the
symmetric group actions.

There is a forgetful functor $U$ from $\Prop$ to $\Mega$, defined by 
\begin{align*}
	(U\T)_0 &= \col \T \\ 
	(U\T)_1 &= 
	\coprod_{
		\substack{\lis{a_i}_{i=1}^n, \lis{b_j}_{j=1}^m \\ \in \M(\col \T)}} 
	\T(a_1, \dots, a_n; b_1, \dots, b_m) 
\end{align*} with
the induced source and target maps. 

\begin{thm}\label{T:freeprop}
The functor $U: \Prop \to \Mega$ has a left adjoint $F: \Mega \to \Prop$.
\end{thm}

The proof of this theorem is contained in appendix~\ref{S:freeprop}. We
would like to note that our construction of $F(\X)$ is necessarily
isomorphic to that in the fixed color setting given in \cite{fmy}, but
we still need to show adjointness in the case where the color sets may
vary and maps need not preserve color. For this purpose we prefer to
have a very explicit description of $F(\X)$.

\section{The category of props is complete and cocomplete}

Limits in $\Prop$ are obtained by taking the corresponding
limits on colors and morphisms in $\Set$. The goal for this section is
thus to show

\begin{thm}\label{T:cocomplete}
The category $\Prop$ is cocomplete.
\end{thm}

Our proof is a minor adaptation of that in \cite[\S 4]{elmendorfmandell} for the category of multicategories. 

Recall that a \emph{permutative category} is a symmetric monoidal
category $\C$ with a strictly associative product $\oplus$, a strict
unit $0$, a swap map $\gamma: a \oplus b \cong b\oplus a$ which has the
\emph{equalities} $\gamma \gamma = \id$, $\gamma = (\gamma \oplus 1)
(1\oplus \gamma)$, and $(a\oplus 0 \overset\gamma\to 0 \oplus a
\overset=\to a) = (a\oplus 0 \overset=\to a)$; see \cite[3.1]{elmendorfmandell1}. A \emph{strict map} $f: \C \to \mcd$ of
permutative categories is a functor with $f(a\oplus b) = fa \oplus fb$,
$f(0)=0$, and \[ [f(a\oplus b) \overset=\to fa \oplus fb
\overset\gamma\to fb \oplus fa]= [f(a\oplus b) \overset{f\gamma}\to f(b
\oplus a) \overset=\to fb \oplus fa].\] Let $\Perm$ be the category of
permutative categories and strict morphisms, which is cocomplete by
\cite[4.1]{elmendorfmandell}. 
There is a functor $U : \Perm \to
\Prop$ which is given on objects by
\begin{align*} 
\col U(\C) &= \ob \C \\ 
	U(\C)(c_1, \dots, c_n; d_1, \dots, d_m) &= 
	\C(c_1 \oplus \dots \oplus c_n, d_1 \oplus \dots \oplus d_m). 
\end{align*}
If $f: \C \to \D$ is a homomorphism of permutative categories, then there is an evident homomorphism of props $U(f): U(\C) \to U(\D)$ given by
\begin{equation*}
	\col U(\C) = \ob \C \to \ob \D = \col U(\D) 
\end{equation*}
\begin{multline*}
		U(\C)(\lis{c_i}_{i=1}^n; \lis{d_j}_{j=1}^m) = 
		\C(c_1 \oplus \dots \oplus c_n, d_1 \oplus \dots \oplus d_m) \\
		\overset{f}\to \D(fc_1 \oplus \dots \oplus fc_n, fd_1 \oplus \dots \oplus fd_m) = U(\D)(\lis{fc_i}_{i=1}^n; \lis{fd_j}_{j=1}^m).
\end{multline*}

\begin{rem}
Notice that \emph{all} props arise in this way -- a prop  is the same
thing as a permutative category $\C$ which has a set of indecomposable
objects $S$ and $\ob \C = \M S$ with $\oplus$ given by
concatenation.\footnote{This unraveling of the definition in the
monochrome case is pointed out in \cite{pirashvili}.} 
\end{rem}

\begin{prop}
The functor $U$ has a left adjoint.
\end{prop}
\begin{proof} The left adjoint $L$ is constructed as follows. If $\Q$ is
a prop, then the objects of $L(\Q)$ are finite lists of colors of $\Q$:
\[ \ob L(\Q) = \M \col(\Q) = \coprod_{k\geq 0} \col(\Q)^{\times k}. \] 
The monoidal product of two lists is
given by concatenation. 
Given two lists $\lis{a_i}_{i=1}^n$ and $\lis{b_j}_{j=1}^m$, we define \[ L(\Q) (\lis{a_i}_{i=1}^n, \lis{b_j}_{j=1}^m) = \Q(\lis{a_i}_{i=1}^n, \lis{b_j}_{j=1}^m). \]
If $f: \Q \to \Q'$ is a prop homomorphism, then we get a homomorphism of permutative categories which on objects is
\[ \ob L(\Q) = \M \col (\Q) \overset{\M f}\longrightarrow \M \col(\Q') = \ob L(\Q') \]
and on morphisms is induced directly from $\Q$.
\end{proof}

\begin{rem} The left adjoint $\Operad \to \Perm$ given in
\cite{elmendorfmandell}  factors through our $L$. Specifically, the main
part of their construction is actually giving the map $F: \Operad \to
\Prop$ from Proposition~\ref{P:operadpropadjunction}. We have a
composition of adjunctions \begin{equation} \Operad
\overset{F}\rightleftarrows \Prop \overset{L}\rightleftarrows \Perm
\label{E:adjunctions}\end{equation} which recovers the adjunction
\cite[4.2]{elmendorfmandell}. \end{rem}

\begin{lem}\label{L:reflectisomorphism}
The left adjoint $L: \Prop \to \Perm$ reflects isomorphisms.
\end{lem}
\begin{proof} Let $\star$ be the terminal object of $\Prop$.
Specifically, $\col (\star) = \set{ 1 }$ is a one element set and
$\star(\underbrace{1, \dots, 1}_n; \underbrace{1, \dots, 1}_m)$ is a one
element set for all $n,m\geq 0$. We identify the
objects of the permutative category $L(\star)$ with 
the set of nonnegative integers $\N$.

Consider the unit $\eta: \id_{\Prop} \Rightarrow UL$. 
\newcommand{\length}{length}
If $\Q$ is a prop then $\eta_\Q: \Q \to UL\Q$ takes a color $a \in \col(\Q)$ to the one-element list $\lis{a} \in \col(UL\Q) = \M \col(\Q)$.
Consider the commutative diagram
\[ \xymatrix{
\Q \ar@{->}[r]^-{\eta_\Q} \ar@{->}[d] & UL\Q \ar@{->}[d] \\
\star \ar@{->}[r]^-{\eta_\star} & UL\star
}\]
and suppose that we have 
\[ \xymatrix@R=.2cm{ f: 
\lis{a^1_i}_{i=1}^{j_1}, \dots, \lis{a^n_i}_{i=1}^{j_n} 
\ar@{->}[r] \ar@{=}[d] &
\lis{b^1_i}_{i=1}^{k_1}, \dots, \lis{b^m_i}_{i=1}^{k_m} \ar@{=}[d] \\
\lis{\lis{a_i^x}_{i=1}^{j_x}}_{x=1}^n \ar@{->}[r] & \lis{\lis{b_i^y}_{i=1}^{k_y}}_{y=1}^m
}\] 
in $UL\Q$. The image of $f$ in $UL\star$ is \[ \bar{f}: \lis{j_x}_{x=1}^n \to \lis{k_y}_{k=1}^m \]
If all of these list lengths $j_x, k_y$ are $1$, then $f$ is in the image of $\eta_\Q$. 
On the other hand, if
$f$ is in the image of $\eta_\Q$, then $\bar f$ is in the image of 
$\eta_\star$,
so all of these lengths must be one.
In other words, the image of $\eta$ is the preimage of the subprop\footnote{cf. Definition~\ref{subprop}} of $UL\star$ which is generated by single object $\set{ 1 }$.

Suppose that $\alpha: \Q \to \Q'$ is a map of props so that $L\alpha$ is
an isomorphism. The diagram
\[ \xymatrix{
\Q \ar@{->}[rr]^\alpha \ar@{->}[d]^\eta & &
\Q' \ar@{->}[d]^\eta \\
UL \Q \ar@{->}[rr]^{UL\alpha}_\cong \ar@{->}[dr] &&
UL \Q' \ar@{->}[dl] \\
&UL \star
}\]
shows that $UL\alpha$ is an isomorphism between the preimages as above.
By construction of $L\Q$, the unit is injective, so $\alpha$ is an
isomorphism.
\end{proof}

\begin{lem}\label{L:preserveequalizer}
The functor $L$ preserves equalizers. 
\end{lem}
\begin{proof} 
As noted in the proof of \cite[4.5]{elmendorfmandell},
equalizers in $\Perm$ are created in $\Cat$. Consider the equalizer
diagrams
\begin{align*} \T' \overset\alpha\to \T & \underset\gamma{\overset\beta\rightrightarrows} \T'' \\
\C \to L\T & \underset{L\gamma}{\overset{L\beta}\rightrightarrows} L\T'',\end{align*}
where the first is in $\Prop$ and the second is in $\Perm$. We wish to
show that $\C = L\T'$; certainly $L\T' \ci \C$. Notice that $\ob \C \ci
\coprod_{k\in \N} (\col \T)^{\times k}$ is the subset consisting of lists
$\lis{a_i}_{i=1}^n$ such that $L\beta \left( \lis{a_i}_{i=1}^n\right) =
L\gamma \left(\lis{a_i}_{i=1}^n\right)$, i.e. $\beta a_i = \gamma a_i$ for all $i$. Thus if $\lis{a_i}_{i=1}^n \in \ob \C$ then so is
$a_i$. But then if $a \in \ob \C$, we have $\beta a = L\beta a = L\gamma
a = \gamma a$, so $\ob \C \ci \ob L\T'$. 

We have shown that $\ob \C = \ob L\T'$, and we know that $L\T' \ci \C
\ci L\T$. Suppose that $f\in \C(\lis{a_i}_{i=1}^n; \lis{b_j}_{j=1}^m)$.  Then \[
f\in L\T(\lis{a_i}_{i=1}^n; \lis{b_j}_{j=1}^m) = \T (a_1, \dots, a_n; b_1, \dots,
b_m) \] has the property that $\beta(f) = L\beta(f) = L\gamma(f) =
\gamma(f)$, so considering $f$ as an element of $\T (a_1, \dots, a_n;
b_1, \dots, b_m)$ we see that it is actually an element of $\T' (a_1,
\dots, a_n; b_1, \dots, b_m)$. 
Thus we have shown that $\C(\lis{a_i}_{i=1}^n; \lis{b_j}_{j=1}^m) \ci L\T' (\lis{a_i}_{i=1}^n; \lis{b_j}_{j=1}^m)$, and conclude that
$\C = L\T'$.
\end{proof}

\begin{proof}[Proof of Theorem~\ref{T:cocomplete}]
We apply the dual of \cite[Ch.3, Theorem 3.14]{ttt} to $L$, using that
$L$ has a right adjoint, $L$ reflects isomorphisms by
Lemma~\ref{L:reflectisomorphism}, $\Prop$ has all equalizers, and $L$
preserves equalizers by Lemma~\ref{L:preserveequalizer}. The cited
theorem then gives that the adjunction \[ L \colon \Prop
\rightleftarrows \Perm \colon U \] is comonadic. In other words, $\Prop$
is equivalent to the category of coalgebras over the comonad $LU$ on
$\Perm$. Cocompleteness of this category of coalgebras follows from that
of $\Perm$ (see  exercise 2 in \cite[VI.2]{maclane}), so $\Prop$ is 
cocomplete.
\end{proof}

%%%%%%%%%%%%%%%%%%%%%%%%%%%%%%%%%%%%%%%%%%%%%%%%%%%%%
\section{A closed symmetric monoidal structure on \texorpdfstring{$\Prop$}{Prop}}
%%%%%%%%%%%%%%%%%%%%%%%%%%%%%%%%%%%%%%%%%%%%%%%%%%%%%

\subsection{\texorpdfstring{$\Prop$}{Prop} is enriched over \texorpdfstring{$\Prop$}{Prop}}\label{S:selfenriched}
Suppose that $\R$ and $\T$ are two props. We define a mapping prop
between them, which we denote by $\inthom (\R, \T)$. The colors of
$\inthom(\R, \T)$ are just prop maps $\R \to \T$. We now must define a
propic natural transformation; to begin, let us take $p+q$ prop maps $f_1,
\dots, f_p, g_1, \dots, g_q$ from $\R$ to $\T$. A \emph{$(p,q)$ natural
transformation} 
\[ \xi: \blis{f_1, \dots, f_p} \Rightarrow \blis{g_1, \dots, g_q}\] 
is a collection of $\T$-morphisms
\[ \xi_a \in \T(f_1 a, \dots, f_pa; g_1 a, \dots, g_qa), 
\] 
one for each $a\in \col \R$. There is, of course,
some consistency condition: if $\phi: \blis{a_1, \dots, a_n} \to \blis{b_1, \dots,
b_m}$ is in $\R$, then the following 
octagon 
must commute.

%%% Better clean octagon for reuse
%\[ \xymatrix{
%			&	\lis{\lis{}_{=1}}_{=1}		\ar@{->}[r] \ar@{->}[dl]_\cong	&	\lis{\lis{}_{=1}}_{=1}		\ar@{->}[dr]^\cong	&				\\
%\lis{\lis{}_{=1}}_{=1}	\ar@{->}[d]		&				&				&	\lis{\lis{}_{=1}}_{=1}		\ar@{->}[d]	\\
%\lis{\lis{}_{=1}}_{=1}		\ar@{->}[dr]_\cong	&				&				&	\lis{\lis{}_{=1}}_{=1}	\ar@{->}[dl]^\cong		\\
%			&		\lis{\lis{}_{=1}}_{=1}	\ar@{->}[r]	&		\lis{\lis{}_{=1}}_{=1}		&				\\
%}\]

\begin{equation} \xymatrix@=+12pt@R=+12pt{
			&	\lis{\lis{f_ja_i}_{j=1}^p}_{i=1}^n		\ar@{->}[r]^{\lis{\xi_{a_i}}} \ar@{->}[dl]_\cong	&	\lis{\lis{g_\ell a_i}_{\ell=1}^q}_{i=1}^n		\ar@{->}[dr]^\cong	&				\\
\lis{\lis{f_ja_i}_{i=1}^n}_{j=1}^p	\ar@{->}[d]_{\lis{f_j\phi}}		&				&				&	\lis{\lis{g_\ell a_i}_{i=1}^n}_{\ell=1}^q		\ar@{->}[d]^{\lis{g_\ell \phi}}	\\
\lis{\lis{f_jb_k}_{k=1}^m}_{j=1}^p		\ar@{->}[dr]_\cong	&				&				&	\lis{\lis{g_\ell b_k}_{k=1}^m}_{\ell=1}^q	\ar@{->}[dl]^\cong		\\
			&		\lis{\lis{f_jb_k}_{j=1}^p}_{k=1}^m	\ar@{->}[r]_{\lis{\xi_{b_k}}}	&		\lis{\lis{g_\ell b_k}_{\ell=1}^q}_{k=1}^m		&				\\
}
\label{DIAGknat}
\end{equation}
This is a convenient abuse of notation which we employ frequently. `Commutativity' of this octagon means precisely that 
\begin{equation}
	\bar \tau_* (\lis{g_\ell \phi}_{\ell=1}^q) \circ_v \tau_* (\lis{\xi_{a_i}}_{i=1}^n) = \bar \sigma^* (\lis{\xi_{b_k}}_{k=1}^m)  \circ_v \sigma^* (\lis{f_j\phi}_{j=1}^p )
\label{meaning-of-DIAGknat}
\end{equation}
where $\sigma,\bar \sigma, \tau,\bar\tau$ are the obvious interchange permutations given by the symbol `$\cong$' in \eqref{DIAGknat} and 
where angular brackets denote \emph{horizontal compositions}, e.g. \[
\lis{\xi_{a_i}} = \xi_{a_1} \circ_h \xi_{a_2} \circ_h \cdots \circ_h
\xi_{a_n}. \] 
We declare that $\inthom(\R, \T)( \lis{f_j}_{j=1}^p, \lis{g_\ell}_{\ell =1}^q)$ be the set of $(p,q)$ natural transformations $\blis{f_1, \dots,
f_p} \Rightarrow \blis{g_1, \dots, g_q}$.

\begin{prop}
The collection of natural transformations $\inthom(\R,\T)$ is a prop. 
\end{prop}
\begin{proof}
Let $\xi$, $\xi'$ be natural transformations, and $f\in
\col(\inthom(\R,\T)) = \Hom(\R,\T)$. We define the prop structure by
defining the maps at each $a\in \col(\R)$:
\begin{align*}
(\xi \circ_v \xi')_a &= \xi_a \circ_v \xi'_a & (\id_f)_a &= \id_{f(a)} \\
(\xi \circ_h \xi')_a &= \xi_a \circ_h \xi'_a & \left[ \sigma^* \tau_* (\xi)\right]_a &= \sigma^* \tau_* \left[ \xi_a \right].
\end{align*}
All of the axioms of the prop then follow directly from the fact that
$\R$ and $\T$ are props. One must show that these actually give natural
transformations, but verifying that diagram \eqref{DIAGknat} commutes
for these various assignments follows by modifying the diagrams for
$\xi$ and $\xi'$.
\end{proof}

If one wishes to show something is a natural transformation, it is often
easier to show that the above diagram commutes on a \emph{generating
set}. We now prove that this is enough.

\begin{defn}\label{D:naturalwrt} Let $f_1, \dots, f_p, g_1, \dots, g_q
\in \Prop(\R, \T)$, $\phi: \lis{a_i}_{i=1}^n \to \lis{b_k}_{k=1}^m$ in
$\R$, and let $\xi$ assign, for each $a\in \col(\R)$, a map \[ \xi_a :
\blis{f_1a, \dots, f_pa} \to \blis{g_1a, \dots, g_qa} \] in $\T$. We say that
\emph{$\xi$ is natural with respect to $\phi$} if the diagram
\eqref{DIAGknat} on page~\pageref{DIAGknat} commutes for the map $\phi$.
If $S$ is a set of maps and $\xi$ is natural with respect to each $\phi
\in S$, then we say \emph{$\xi$ is natural with respect to $S$}.
\end{defn}

\begin{lem}Let 
$\xi$ be as in Definition~\ref{D:naturalwrt}. 
If $\xi$ is natural with respect to a composable pair $\phi$ and $\psi$,
then $\xi$ is natural with respect to $\phi \circ_v \psi$.
\end{lem}
\begin{proof}
\[ \xymatrix{
			&	\lis{\lis{f_ja_i}_{j=1}^p}_{i=1}^n		\ar@{->}[r]^{\lis{\xi_{a_i}}} \ar@{->}[dl]_\cong	&	\lis{\lis{g_\ell a_i}_{\ell=1}^q}_{i=1}^n		\ar@{->}[dr]^\cong	&				\\
\lis{\lis{f_ja_i}_{i=1}^n}_{j=1}^p	\ar@{->}[d]_{\lis{f_j\psi}}		&				&				&	\lis{\lis{g_\ell a_i}_{i=1}^n}_{\ell=1}^q		\ar@{->}[d]^{\lis{g_\ell \psi}}	\\
\lis{\lis{f_jb_k}_{k=1}^m}_{j=1}^p		\ar@{->}[d]_{\lis{f_j\phi}}	 \ar@{->}[r]^\cong & \lis{\lis{f_jb_k}_{j=1}^p}_{k=1}^m	\ar@{->}[r]_{\lis{\xi_{b_k}}}	&		\lis{\lis{g_\ell b_k}_{\ell=1}^q}_{k=1}^m	 \ar@{<->}[r]^\cong	&	\lis{\lis{g_\ell b_k}_{k=1}^m}_{\ell=1}^q	\ar@{->}[d]^{\lis{g_\ell \phi}}		\\
\lis{\lis{f_jc_h}_{h=1}^r}_{j=1}^p		\ar@{->}[dr]_\cong	&				&				&	\lis{\lis{g_\ell c_h}_{h=1}^r}_{\ell=1}^q	\ar@{->}[dl]^\cong		\\
			&		\lis{\lis{f_jc_h}_{j=1}^p}_{h=1}^r	\ar@{->}[r]_{\lis{\xi_{c_h}}}	&		\lis{\lis{g_\ell c_h}_{\ell=1}^q}_{h=1}^r		&				\\
}
\]

\end{proof}

\begin{lem}Let 
$\xi$ be as in Definition~\ref{D:naturalwrt}. 
If $\xi$ is natural with respect to $\phi$ and $\psi$, then $\xi$ is
natural with respect to $\phi \circ_h \psi$.
\end{lem}
\begin{proof}
Let
\begin{align*}
\phi: \blis{a_1, \dots, a_n} &\to \blis{b_1, \dots, b_m} \\
\psi: \blis{a_{n+1}, \dots, a_{n'}} &\to \blis{b_{m+1}, \dots, b_{m'}}
\end{align*}
In figure~\ref{F:naturalityhorizontal},
\begin{figure}
\rotatebox{90}{%
\parbox{8in}{%
\[\xymatrix{
			&	\lis{\lis{f_ja_i}_{j=1}^p}_{i=1}^{n'}		\ar@{->}[r]^{\lis{\xi_{a_i}}} \ar@{->}[dl]_\cong	\ar@{->}[d]_\cong &	\lis{\lis{g_\ell a_i}_{\ell=1}^q}_{i=1}^{n'}		\ar@{->}[dr]^\cong	\ar@{->}[d]^\cong &				\\
\lis{\lis{f_ja_i}_{i=1}^{n'}}_{j=1}^p	\ar@{->}[d]_{\lis{(f_j\phi)\circ_h (f_j\psi)}}		&		\lis{\lis{f_ja_i}_{i=1}^n}_{j=1}^p, \lis{\lis{f_ja_i}_{i=n+1}^{n'}}_{j=1}^p \ar@{->}[d]^{\lis{f_j\phi} \circ_h \lis{f_j\psi}}  \ar@{->}[l]_-\cong 	&	\ar@{->}[d]^{\lis{g_\ell \phi} \circ_h \lis{g_\ell \psi}} 	\lis{\lis{g_\ell a_i}_{i=1}^n}_{\ell=1}^q, \lis{\lis{g_\ell a_i}_{i=n+1}^{n'}}_{\ell=1}^q	 \ar@{->}[r]^-\cong	&	\lis{\lis{g_\ell a_i}_{i=1}^{n'}}_{\ell=1}^q		\ar@{->}[d]^{\lis{(g_\ell \phi) \circ_h (g_\ell \psi)}}	\\
\lis{\lis{f_jb_k}_{k=1}^{m'}}_{j=1}^p		\ar@{->}[dr]_\cong \ar@{->}[r]^-\cong 	&		\lis{\lis{f_jb_k}_{k=1}^m}_{j=1}^p, \lis{\lis{f_jb_k}_{k=m+1}^{m'}}_{j=1}^p \ar@{->}[d]_\cong 	&	\lis{\lis{g_\ell b_k}_{k=1}^m}_{\ell=1}^q, \lis{\lis{g_\ell b_k}_{k=m+1}^{m'}}_{\ell=1}^q		\ar@{->}[d]^\cong 	&	\lis{\lis{g_\ell b_k}_{k=1}^{m'}}_{\ell=1}^q	\ar@{->}[dl]^\cong	\ar@{->}[l]_-\cong	\\
			&		\lis{\lis{f_jb_k}_{j=1}^p}_{k=1}^{m'}	\ar@{->}[r]_{\lis{\xi_{b_k}}}	&		\lis{\lis{g_\ell b_k}_{\ell=1}^q}_{k=1}^{m'}		&				\\
}\]
} % end of parbox
} % end of rotatebox
\caption{Naturality of horizontal composition}\label{F:naturalityhorizontal}
\end{figure}
the middle rectangle commutes since it is the horizontal composition of
two octagons which commute, using that $\xi$ is natural with respect to $\phi$
and with respect to $\psi$. The left and right squares commute by
\eqref{E:horizswap}. 
\end{proof}

\begin{lem} Let 
$\xi$ be as in Definition~\ref{D:naturalwrt}. If $\xi$ is natural with respect to $\phi$, then $\xi$ is
natural with respect to $\sigma^* \phi$ and $\tau_* \phi$. \end{lem}
\begin{proof} We will show that $\xi$ is natural with respect to $\sigma^* \tau_* \phi$.

Notice that \eqref{E:vertcompat} implies that $ \sigma^* \phi = \phi
\circ_v (\sigma^* \id) $ and $(\tau_* \id) \circ_v \phi = \tau_* \phi$.
In the current setting we thus have a commutative diagram
\[ \xymatrix{
\lis{f_j a_{\sigma(i)}}_{i=1}^n \ar@{->}[r]^{\sigma^* \id} \ar@{->}[d]^{f_j(\sigma^*\tau_* \phi)}
& \lis{f_j a_i}_{i=1}^n \ar@{->}[d]^{f_j \phi}\\
\lis{f_j b_{\tau^\inv (k)}}_{k=1}^m  & \lis{f_j b_k}_{k=1}^m \ar@{->}[l]^-{\tau_*\id}
}\]
for each $j$, where the top and bottom maps are isomorphisms. This
diagram also commutes if we replace $f_j$ by $g_\ell$.

We then have a commutative diagram
\[\xymatrix{
\lis{\lis{f_ja_{\sigma(i)}}_{j=1}^p}_{i=1}^{n}	\ar@{->}[r]^-{block}_-\cong \ar@{->}[d]^\cong	
&	\lis{\lis{f_ja_i}_{j=1}^p}_{i=1}^{n}		\ar@{->}[r]^{\lis{\xi_{a_i}}} 	\ar@{->}[d]_\cong 
&	\lis{\lis{g_\ell a_i}_{\ell=1}^q}_{i=1}^{n}		\\
\lis{\lis{f_ja_{\sigma(i)}}_{i=1}^{n}}_{j=1}^p	\ar@{->}[d]_{\lis{ f_j (\sigma^*\tau_* \phi)}}	 \ar@{->}[r]^{\lis{\sigma^* \id}}
&		\lis{\lis{f_ja_i}_{i=1}^n}_{j=1}^p  \ar@{->}[d]^{\lis{f_j\phi} }  	
&	
\\
\lis{\lis{f_jb_{\tau^\inv (k)}}_{k=1}^{m}}_{j=1}^p		\ar@{->}[d]^\cong 
&		\lis{\lis{f_jb_k}_{k=1}^m}_{j=1}^p \ar@{->}[d]_\cong 	\ar@{->}[l]_-{\lis{\tau_* \id}}
&	
\\
\lis{\lis{f_jb_{\tau^\inv (k)}}_{j=1}^{p}}_{k=1}^m \ar@{->}[r]^-{block}_-\cong
&		\lis{\lis{f_jb_k}_{j=1}^p}_{k=1}^{m}	\ar@{->}[r]_{\lis{\xi_{b_k}}}	
&		\lis{\lis{g_\ell b_k}_{\ell=1}^q}_{k=1}^{m}		
&				
\\
}\]
as well as a similar one for the $g_\ell$. These glue together along the
commutative octagon \eqref{DIAGknat} for $\phi$. The resulting large
commutative diagram shows that the octagon for $\sigma^*\tau_* \phi$
commutes.
\end{proof}

\begin{defn}\label{subprop}Suppose that $S$ is a set of morphisms in a prop $\T$. Then
the \emph{subprop generated by $S$}, denoted $\lis{S}$, is the smallest
subprop of $\T$ containing all elements of $S$. \end{defn}

This subprop $\lis{S}$ must contain all of the identity maps on the
colors appearing in the source and target lists of elements of $S$. It
must also contain all morphisms obtained by iterated compositions and
symmetric group actions from elements of $S$ and these identity maps.
The collection of such morphisms forms a prop, as we saw in the
construction of the free prop in appendix~\ref{S:freeprop}. Therefore we
have an alternate characterization of $\lis{S}$.

\begin{prop}\label{P:generating} Let $S$ be a set of morphisms in $\R$
and let $\lis{S}$ be the subprop of $\R$ generated by $S$. 
If $\xi$ is natural with respect to $S$ then $\xi$ is
natural with respect to $\lis{S}$. In particular, if $\R$ is generated
by $S$, then $\xi$ is a natural transformation. \end{prop}
\begin{proof}
This follows from the preceding paragraph and the three preceding lemmas.
\end{proof}

We would hope that this enrichment be compatible with the existing
enrichment on the category of operads. We cannot insist that the
adjunction be enriched, since the categories are enriched over different
things. We do have

\begin{prop}\label{P:operadadjunction} Given the adjunction $ F \colon
\Operad \rightleftarrows \Prop \colon U, $ the isomorphism
$\Hom_{\Prop}(F(\someoperad), \T) \cong \Hom_{\Operad}(\someoperad,
U(\T))$ extends to an isomorphism of operads \[
U(\inthom(F(\someoperad), \T)) \cong \inthom_{\Operad}(\someoperad,
U(\T)) \] where the right hand side is the internal hom in $\Operad$.
\end{prop}
\begin{proof} Suppose $\xi \in \inthom_{\Operad}(\someoperad, U(\T))$ is
a $p$-natural transformation $\blis{f_1, \dots, f_p} \to g$, where $f_1,
\dots, f_p, g$ are operad maps $\someoperad \to U(\T)$. This means that
we have maps $\xi_a: f_1a, \dots, f_pa \to g(a)$ in $U(\T)$ for each
$a\in \col \someoperad$ which satisfy the compatibility condition from
\cite[2.2]{elmendorfmandell}.

Let $\bar f_1, \dots, \bar f_p, \bar g$ be the adjoints of the above
operad maps and note that $\bar f a = f a$ for all $a\in
\col(\someoperad)=\col(F\someoperad)$.
Furthermore, the maps \[ \xi_a: \blis{\bar f_1a, \dots, \bar f_pa} \to \bar
g\] are already in $\T$ itself by the definition of $U$. We define the
adjoint \[ \bar \xi: \blis{\bar f_1, \dots, \bar f_p} \to \bar g \] to be
\[ \bar \xi_a = \xi_a: \blis{\bar f_1a, \dots, \bar f_pa} \to  \bar ga. \]

Notice that the class of maps $\phi: \blis{a_1, \dots, a_n} \to b$ generate
$F(\someoperad)$, so by Proposition~\ref{P:generating} it is enough to
show that \eqref{DIAGknat} commutes for $\phi \in F(\someoperad)$ with
one output. Now we are in the situation with $q=1$ and $m=1$, so the
octagon becomes
\[ \xymatrix{
\lis{\lis{ \bar f_j a_i}_{j=1}^p}_{i=1}^n \ar@{->}[r]^-{\lis{\bar \xi_{a_i}}} \ar@{->}[d]_\cong
& \lis{\bar g a_i}_{i=1}^n \ar@{->}[dd]^{\bar g\phi} \\
\lis{ \lis{  \bar f_j a_i   }_{i=1}^n }_{j=1}^p \ar@{->}[d]_{\lis{ \bar f_j \phi }}  & 
\\
\lis{ f_j b }_{j=1}^p  \ar@{->}[r]_{\bar\xi_{b}}
  & \bar{g} b
}\]
But since $\phi$ only has one output and $UF=\id$ by
Proposition~\ref{P:UFequalsID}, $\phi$ is actually in $\someoperad$ so
$\bar f_j \phi \in \T$ is actually in $U(\T)$, and equals $f_j \phi$.
Thus the diagram above is  a diagram in $U(\T)$, and is exactly the
diagram from \cite[2.2]{elmendorfmandell}, so commutes.

Conversely, suppose that we have a $(p,1)$-natural transformation \[
\xi: \blis{f_1, \dots, f_p} \to g \] in $\inthom(F(\someoperad), \T)$,
where $f_1, \dots, f_p, g$ are prop maps from $F(\someoperad) \to \T$.
The collection of $(p,1)$-natural transformations constitute the set of
arrows of $U(\inthom(F(\someoperad), \T))$. Let $\bar f_1, \dots, \bar
f_p, \bar g: \someoperad \to U(\T)$ be the adjoints of $f_1, \dots, f_p,
g$. Since $\xi_a \in \T$ has only one output, it is actually in $U(\T)$.
We thus define $\bar \xi_a: \lis{\bar f_ja}_{j=1}^p \to  \bar
ga$ to be $U(\xi_a)$. Let $\phi: \blis{a_1, \dots, a_n} \to b$ in
$\someoperad$. Then the octagon \eqref{DIAGknat} commutes for $F\phi$
and $\xi$. Applying $U$, we get the commutative diagram from
\cite[2.2]{elmendorfmandell}, so $\bar\xi = (\bar\xi_a)$ is a
$p$-natural transformation. \end{proof}

\subsection{Bilinear maps of props} 
Suppose that $\R, \SSS$, and $\T$ are props. A bilinear map $(\R, \SSS) \to
\T$ consists of the following data.
\begin{enumerate}
\item a function $\chi: \col \R \times \col \SSS \to \col \T$
\item for each $\phi: \blis{a_1, \dots, a_n} \to \blis{b_1, \dots, b_m}$ in $\R$
and $c \in \col \SSS$, a morphism
\[ \chi(\phi, c): \blis{\chi(a_1, c), \dots, \chi(a_n,c)} \to \blis{\chi(b_1, c),
\dots, \chi(b_m,c)} \] in $\T$.
\item for each $\psi: \blis{c_1, \dots, c_p} \to \blis{d_1, \dots, d_q}$ in $\SSS$
and $a\in \col \R$, a morphism \[ \chi(a, \psi): \blis{\chi(a, c_1), \dots,
\chi(a,c_p)} \to \blis{\chi(a, d_1), \dots, \chi(a,d_q)} \] in $\T$.
\end{enumerate}
These are required to satisfy the axioms
\begin{enumerate}
\item if $a\in \col \R$ then $\chi(a,-)$ is a map of props $\SSS \to \T$,
\item if $c\in \col \SSS$ then $\chi(-,c)$ is a map of props $\R \to \T$, and
\item the octagon
\begin{equation}\label{bilinear-octagon} \xymatrix{
			&	\lis{\lis{\chi(a_i ,c_j )}_{i=1}^n}_{j=1}^p		\ar@{->}[r]^{\lis{\chi(\phi, c_j)}} \ar@{->}[dl]_\cong	&	\lis{\lis{\chi(b_k ,c_j )}_{k=1}^m}_{j=1}^p		\ar@{->}[dr]^\cong	&				\\
\lis{\lis{\chi(a_i ,c_j )}_{j=1}^p}_{i=1}^n	\ar@{->}[d]_{\lis{\chi(a_i,\psi)}}		&				&				&	\lis{\lis{\chi(b_k ,c_j )}_{j=1}^p}_{k=1}^m		\ar@{->}[d]^{\lis{\chi(b_k,\psi)}}	\\
\lis{\lis{\chi(a_i ,d_\ell )}_{\ell=1}^q}_{i=1}^n		\ar@{->}[dr]_\cong	&				&				&	\lis{\lis{\chi(b_k ,d_\ell )}_{\ell=1}^q}_{k=1}^m	\ar@{->}[dl]^\cong		\\
			&		\lis{\lis{\chi(a_i ,d_\ell )}_{i=1}^n}_{\ell=1}^q	\ar@{->}[r]_{\lis{\chi(\phi,d_\ell)}}	&		\lis{\lis{\chi( b_k, d_\ell)}_{k=1}^m}_{\ell=1}^q		&				\\
}\end{equation}
commutes.
\end{enumerate}

We will write $\bilinue(\R,\SSS; \T)$ for the collection of bilinear
maps. Unravelling the definitions gives natural bijections
\begin{equation} \Hom(\R, \inthom(\SSS, \T)) \cong \bilinue(\R,\SSS; \T)
\cong \Hom(\SSS, \inthom(\R, \T)). \label{E:unenrichedbilinear}
\end{equation}

We would like to show that the collection of bilinear maps is the color
set for a prop $\bilin(\R, \SSS; \T)$. To this end, suppose that we have
a list \[ \chi_1, \dots, \chi_v, \varsigma_1, \dots, \varsigma_w
\] of bilinear maps $(\R, \SSS) \to \T$. A $(v,w)$-morphism \[ \xi:
\blis{\chi_1, \dots, \chi_v} \to \blis{\varsigma_1, \dots, \varsigma_w} \] in
$\bilin(\R, \SSS; \T)$ consists of a choice of $(v,w)$ morphisms \[
\xi_{(a,c)}: \blis{\chi_1(a,c), \dots, \chi_v(a,c)} \to \blis{\varsigma_1(a,c),
\dots, \varsigma_w(a,c)} \] for each $a\in \col(\R)$ and $c\in
\col(\SSS)$. These are subject to two compatibility conditions. The
first is that if we have
\begin{align*}
\phi&\in \R(a_1, \dots, a_n; b_1, \dots, b_m) \\
c&\in \col(\SSS)
\end{align*}
then the octagon
\[ \xymatrix{
			&	\lis{\lis{\chi_j(a_i ,c )}_{i=1}^n}_{j=1}^v		\ar@{->}[r]^{\lis{\chi_j(\phi, c)}} \ar@{->}[dl]_\cong	&	\lis{\lis{\chi_j(b_\ell ,c )}_{\ell=1}^m}_{j=1}^v		\ar@{->}[dr]^\cong	&				\\
\lis{\lis{\chi_j(a_i ,c )}_{j=1}^v}_{i=1}^n	\ar@{->}[d]_{\lis{\xi_{(a_i,c)}}}		&				&				&	\lis{\lis{\chi_j(b_\ell ,c )}_{j=1}^v}_{\ell=1}^m		\ar@{->}[d]^{\lis{\xi_{(b_\ell,c)}}}	\\
\lis{\lis{\varsigma_k(a_i ,c )}_{k=1}^w}_{i=1}^n		\ar@{->}[dr]_\cong	&				&				&	\lis{\lis{\varsigma_k(b_\ell ,c )}_{k=1}^w}_{\ell=1}^m	\ar@{->}[dl]^\cong		\\
			&		\lis{\lis{\varsigma_k(a_i ,c )}_{i=1}^n}_{k=1}^w	\ar@{->}[r]_{\lis{\varsigma_k(\phi,c)}}	&		\lis{\lis{\varsigma_k(b_\ell ,c )}_{\ell=1}^m}_{k=1}^w		&				\\
}\]
commutes. Similarly, if
\begin{align*}
\psi&\in\SSS(c_1, \dots, c_p; d_1, \dots, d_q)  \\
a&\in \col(\R)
\end{align*}
then the octagon
\[ \xymatrix{
			&	\lis{\lis{\chi_j(a ,c_i )}_{i=1}^p}_{j=1}^v		\ar@{->}[r]^{\lis{\chi_j(a, \psi)}} \ar@{->}[dl]_\cong	&	\lis{\lis{\chi_j(a ,d_\ell )}_{\ell=1}^q}_{j=1}^v		\ar@{->}[dr]^\cong	&				\\
\lis{\lis{\chi_j(a ,c_i )}_{j=1}^v}_{i=1}^p	\ar@{->}[d]_{\lis{\xi_{(a,c_i)}}}		&				&				&	\lis{\lis{\chi_j(a ,d_\ell )}_{j=1}^v}_{\ell=1}^q		\ar@{->}[d]^{\lis{\xi_{(a,d_\ell)}}}	\\
\lis{\lis{\varsigma_k(a ,c_i )}_{k=1}^w}_{i=1}^p		\ar@{->}[dr]_\cong	&				&				&	\lis{\lis{\varsigma_k(a ,d_\ell )}_{k=1}^w}_{\ell=1}^q	\ar@{->}[dl]^\cong		\\
			&		\lis{\lis{\varsigma_k(a ,c_i )}_{i=1}^p}_{k=1}^w	\ar@{->}[r]_{\lis{\varsigma_k(a,\psi)}}	&		\lis{\lis{\varsigma_k(a ,d_\ell )}_{\ell=1}^q}_{k=1}^w		&				\\
}\]
commutes.

\begin{prop}\label{P:bilinenriched} With these morphisms $\bilin(\R,
\SSS; \T)$ is a prop and \[ \inthom(\R, \inthom(\SSS, \T)) \cong
\bilin(\R, \SSS; \T) \cong \inthom(\SSS, \inthom(\R,\T)).\] \qed
\end{prop}

\subsection{The tensor product of props}
If $X$ is an algebra over two props $\T$ and $\T'$, we might ask for
some sort of compatibility of the actions of $\T$ and $\T'$. We can
interchange the actions in a suitable, reasonable sense if and only if
$X$ is an algebra over another prop $\T \otimes \T'$. We describe this
prop in this section as a universal bilinear target $(\T, \T') \to \T
\otimes \T'$. 

Let $\R$ and $\SSS$ be (small) props and consider the coproducts \[
\coprod_{\col(\R)} \SSS \qquad \& \qquad \coprod_{\col(\SSS)} \R. \]
Maps from $\coprod_{\col(\R)} \SSS \to \T$ can be thought of as
$\R$-parametrized maps $\SSS \to \T$. Each consists of a function $\chi:
\col(\R) \times \col(\SSS) \to \col(\T)$ together with a function
$\col(\R) \to \Hom(\SSS,\T)$ which extends $\chi$'s adjoint $\col(\R)
\to \Hom(\col(\SSS), \col(\T))$. Similarly maps $\coprod_{\col(\SSS)} \R
\to \T$ are the same as $\SSS$-parametrized maps $\R \to \T$.

We observe that by forgetting structure, a bilinear map $(\R,\SSS) \to
\T$ may be thought of as either an $\SSS$-parametrized map or as an
$\R$-parametrized map. In addition, these two parametrized maps share
the same function $\chi:\col(\R) \times \col(\SSS)\to\col(\T)$.

We take the pushout 
\[ \xymatrix{
\col(\R) \times \col (\SSS) \ar@{->}[r] \ar@{->}[d] & \coprod_{\col(\R)} \SSS \ar@{->}[d] \\
\coprod_{\col(\SSS)} \R \ar@{->}[r] & \R \# \SSS
}\]
where the upper left corner is the free prop on this set of colors. Thus
maps from $\R \# \SSS$ are the same thing as a pair of parametrized maps
which share the same $\chi$.

Fix morphisms $\phi: \blis{a_1, \dots, a_n} \to \blis{b_1, \dots, b_m}$ in $\R$
and $\psi: \blis{c_1, \dots, c_p} \to \blis{d_1, \dots, d_q}$ in $\SSS$, and
define megagraphs $\X(\phi,\psi)$ and $\Y(\phi,\psi)$ as follows. We
declare \[ X_0 = Y_0 = (\set{a_1, \dots, a_n} \times \set{c_1, \dots,
c_p}) \cup (\set{b_1, \dots, b_m} \times \set{d_1, \dots, d_q}), \] and
let $X_1$ be a two point set $\set{*_1, *_2}$ and $Y_1 = *$. Both source
functions evaluate to $\lis{\lis{(a_i ,c_j )}_{i=1}^n}_{j=1}^p$ and both
target functions evaluate to $\lis{\lis{( b_k, d_\ell)}_{k=1}^m}$. We
consider the map of megagraphs $\X(\phi, \psi) \to \Y(\phi, \psi)$ which
is the identity map on objects and the map of megagraphs $\X(\phi, \psi)
\to U(\R \# \SSS)$ which takes $*_1$ to the composite
\begin{multline*} \lis{\lis{(a_i ,c_j )}_{i=1}^n}_{j=1}^p		\overset{\lis{(\phi, c_j)}}\longrightarrow \lis{\lis{(b_k ,c_j )}_{k=1}^m}_{j=1}^p \overset\cong\longrightarrow 	\lis{\lis{(b_k ,c_j )}_{j=1}^p}_{k=1}^m	 \\	\overset{\lis{(b_k,\psi)}}\longrightarrow	
\lis{\lis{(b_k ,d_\ell )}_{\ell=1}^q}_{k=1}^m	
\overset\cong\longrightarrow \lis{\lis{( b_k, d_\ell)}_{k=1}^m}_{\ell=1}^q
\end{multline*}
and $*_2$ to the composite
\begin{multline*}
\lis{\lis{(a_i ,c_j )}_{i=1}^n}_{j=1}^p \overset\cong\longrightarrow \lis{\lis{(a_i ,c_j )}_{j=1}^p}_{i=1}^n \overset{\lis{(a_i,\psi)}}\longrightarrow \lis{\lis{(a_i ,d_\ell )}_{\ell=1}^q}_{i=1}^n \\
\overset\cong\longrightarrow \lis{\lis{(a_i ,d_\ell )}_{i=1}^n}_{\ell=1}^q	\overset{\lis{(\phi,d_\ell)}}\longrightarrow \lis{\lis{( b_k, d_\ell)}_{k=1}^m}_{\ell=1}^q.
\end{multline*}
These are the two paths around the octagon \eqref{bilinear-octagon} in the definition of bilinear map. 

We define $\R \otimes \SSS$ to be the pushout
\[ \xymatrix{
\coprod_{(\phi, \psi)} F\X(\phi,\psi) \ar@{->}[r] \ar@{->}[d] & \R \# \SSS \ar@{->}[d] \\
\coprod_{(\phi, \psi)} F\Y(\phi,\psi) \ar@{->}[r] & \R \otimes \SSS }
\]
where $F\X(\phi, \psi) \to \R \# \SSS$ is the adjoint of the previously
defined map.

The next theorem follows from the construction of $\R \otimes \SSS$.
\begin{thm} If $\R$ and $\SSS$ are props then there is a bilinear map
$(\R, \SSS) \to \R \otimes \SSS$ which is universal among bilinear maps
from $(\R, \SSS)$. In other words, this map induces a natural
isomorphism \[ \bilinue(\R,\SSS; \T) \cong \Hom(\R \otimes \SSS, \T). \]
\end{thm}

An easy consequence of this and \eqref{E:unenrichedbilinear} is that
\begin{equation} \Hom(\R \otimes \SSS, \T) \cong \Hom(\R, \inthom(\SSS,
\T)) \cong \Hom(\SSS, \inthom(\R, \T)). \label{E:unenrichedhomtensor}
\end{equation}

Let $\chi: (\R, \SSS) \to \R \otimes \SSS$ be the universal bilinear
map. For $\phi$ a morphism in $\R$ and $c$ in $\col(\SSS)$, we write
$\phi \otimes c = \chi(\phi, c)$. Similarly, $a \otimes \psi =
\chi(a,\psi)$ for $a\in \col(\R)$ and $\psi$ a morphism of $\SSS$.

\begin{prop}\label{P:tensorgenerated} The set of morphisms \[ S =
\setm{\phi\otimes c, a \otimes \psi}{a\in \col(\R), c\in \col(\SSS),
\phi \in \R, \text{ and } \psi \in \SSS} \] generate $\R\otimes \SSS$.
Moreover, the set of $(p,q)$-natural transformations $\xi$ in
$\inthom(\R\otimes \SSS, \T)$ are precisely the set of those $\xi$ which
are natural with respect to $S$ (as in definition~\ref{D:naturalwrt}).
\end{prop}
\begin{proof} The subprop generated by $S$ is universal for bilinear
maps. The second statement is Proposition~\ref{P:generating}. \end{proof}

\begin{prop}\label{P:enrichedhomtensor} There is a natural isomorphism
of props \[ \inthom(\R \otimes \SSS, \T) \cong \inthom(\R, \inthom(\SSS,
\T)) \] whose restriction to color sets is
\eqref{E:unenrichedhomtensor}. \end{prop}
\begin{proof} There are isomorphisms
\[  \inthom(\R \otimes \SSS, \T) \cong \bilin(\R,\SSS; \T) 
\cong \inthom(\R, \inthom(\SSS, \T)) \]
given by Propositions~\ref{P:tensorgenerated} and \ref{P:bilinenriched}, respectively.
\end{proof}

\begin{thm}\label{T:csm}
The tensor product $\otimes$ makes $\Prop$ a closed symmetric monoidal category.
\end{thm}
\begin{proof} Symmetry is clear from construction, the unit axioms are
obvious, and the fact that $-\otimes \SSS$ is adjoint to $\inthom(\SSS,
-)$ is \eqref{E:unenrichedhomtensor}. Here the unit is the monochrome
prop and exactly one morphism. As in \cite{elmendorfmandell}, the
associativity isomorphisms come from the isomorphisms
\begin{align*}
\Hom((\R\otimes \SSS) \otimes \T, \mathcal{U}) &\underset{\eqref{E:unenrichedhomtensor}}\cong \Hom(\R \otimes \SSS, \inthom(\T, \mathcal{U})) \\
&\underset{\eqref{E:unenrichedhomtensor}}\cong \Hom(\R, \inthom(\SSS, \inthom(\T, \mathcal{U}))) \\
&\,\,\underset{\ref{P:enrichedhomtensor}}\cong \Hom(\R, \inthom(\SSS \otimes \T, \mathcal{U})) \\
&\underset{\eqref{E:unenrichedhomtensor}}\cong \Hom(\R \otimes(\SSS \otimes \T), \mathcal{U})
\end{align*}
via the Yoneda lemma. The pentagon relation for this associativity
isomorphism follows using the same diagrams one must draw to see the
analogous result in the case of operads in \cite[\S 4]{elmendorfmandell}.
\end{proof}

A natural question is whether this tensor product is `compatible' with
the usual Boardman-Vogt tensor product on operads. We have the following:
\begin{prop}
If $\someoperad, \mathcal{P} \in \Operad$, then
\[ F(\someoperad \otimes_{BV} \mathcal{P}) \cong F(\someoperad) \otimes F(\mathcal{P}). \]
\end{prop}
\begin{proof}
This is a straightforward computation using several natural isomorphisms.
We have, for any prop $\T$,
\begin{align*} \Hom_{\Prop} (F(\someoperad \otimes_{BV} \mathcal{P}), \T) &= 
\Hom_{\Operad} (\someoperad \otimes_{BV} \mathcal{P}, U(\T)) \\
&= \Hom_{\Operad} (\someoperad, \inthom_{\Operad}(\mathcal{P}, U(\T)).
\end{align*}
By Proposition~\ref{P:operadadjunction}, we know that
$\inthom_{\Operad}(\mathcal{P}, U(\T))$ is isomorphic to
$U(\inthom(F(\mathcal{P}), \T))$, so
\begin{align*} &\Hom_{\Operad} (\someoperad, \inthom_{\Operad}(\mathcal{P}, U(\T))) \\
=& \Hom_{\Prop}(\someoperad, U(\inthom(F(\mathcal{P}), \T))) \\
=& \Hom_{\Prop}(F(\someoperad), \inthom(F(\mathcal{P}), \T)) \\
=&\Hom_{\Prop} (F(\someoperad) \otimes F(\mathcal{P}), \T).
\end{align*}
\end{proof}

\appendix

\section{The free prop on a megagraph}\label{S:freeprop}

Fix a symmetric megagraph $\X$. Let $G=(E,V,s,t)$ be a graph.
\begin{defn} A \emph{decoration of $G$ by $\X$} consists of the following data:
\begin{itemize}
\item a function $D_0: E\to X_0$,
\item a function $D_1: V\to X_1$,
\item for each vertex $v \in V$, an ordering on the input 
edges $\tin(v)$ and the output edges $\tout(v)$;
\item an ordering on both $\tin(G)$ and on $\tout(G)$,
\end{itemize} 
which are subject to the compatibility conditions
\[ \xymatrix{
V \ar@{->}[r]^{D_1} \ar@{->}[d]_{in } & X_1 \ar@{->}[d]^s &
V \ar@{->}[r]^{D_1} \ar@{->}[d]_{out } & X_1 \ar@{->}[d]^t \\
\M E \ar@{->}[r]_{\M D_0} & \M X_0 & 
\M E \ar@{->}[r]_{\M D_0} & \M X_0. & 
}\]
\end{defn}

The reader will notice that by specifying that there be an ordering on $in  (v)$ (respectively, an ordering on: $out  (v)$,
$\tin(G)$, or $\tout(G)$) we can consider the set $\tin(v)$ (respectively, $\tout(v)$, $\tin(G)$ or $\tout(G)$)  as an element of
$\M E$. We will use this fact frequently below.

We also remark that in our definition of decoration, the symmetric megagraph $\X$ is fixed. As such, we will usually refer to a
decoration of $G$ by $\X$ as a \emph{decoration of $G$}, or sometimes just a \emph{decoration}. We will denote decorations by variants of
`$\mathfrak{g}$'.

Let us choose a single graph from each isomorphism class and let $\widetilde{\Gamma}$ denote the set of all decorations of all chosen graphs by the symmetric megagraph $\X$.
We define an equivalence relation on $\widetilde{\Gamma}$ using graph automorphisms as follows. Let  $G$ be a graph, let $\mathfrak{g}$, $\widetilde{\mathfrak{g}}$ be two decorations of $G$, and let $f: G\to G$ be a graph automorphism. We say that $f$ \emph{relates} $\mathfrak{g}$ and $\widetilde{\mathfrak{g}}$ if:

\begin{enumerate}
\item $\M f_0(\tin(G)) = \widetilde{in}(G)$;
\item $\M f_0(\tout(G)) = \widetilde{out}(G)$;
\item if both
	\begin{enumerate}
	\item $f_0(\tout(v)) = \sigma \cdot \widetilde{out} (f_1(v))$ and 
	\item $f_0(\tin(v)) = \widetilde{in} (f_1(v)) \cdot \tau$,
	\end{enumerate} 
	then $D_1(v) = \sigma \cdot \widetilde{D}_1(f_1(v)) \cdot \tau$.
\end{enumerate}

The identity automorphism $\id:G\to G$ produces two special relations that we will employ frequently in what follows.

\begin{description} 
\item[Interior permutations] If $I=\tin(v) \cap \tout(v')$ and
$\gamma \in \Sigma_I$ (considered as a subgroup of both
$\Sigma_{\tin(v)}$ and $\Sigma_{\tout(v')}$), then a
decoration $\mathfrak{g}$ is related to a modified decoration
$\widetilde{\mathfrak{g}}$ where the only changes are
\begin{itemize} 
\item $\widetilde{D}_1(v) = (D_1 v) \cdot \gamma$ and $in  (v)$ is
replaced by $\tin(v) \cdot \gamma$
\item $\widetilde{D}_1(v') = \gamma^\inv \cdot (D_1 v')$ and $out 
(v')$ is replaced by $\gamma^\inv \cdot \tout(v')$
\end{itemize}

\item[Exterior permutations] Let $\mathfrak{g}$ be a decoration and
$I\ci \tin(G) \cap \tin(v)$ be a subset with the induced ordering from
$\tin(v)$. 
Let $\gamma$ be such that $I\cdot \gamma^\inv \ci \tin(G)$ is an
ordered inclusion. 
Then $\mathfrak{g}$ is related to a decoration
$\widetilde{\mathfrak{g}}$ where the only changes are
$\widetilde{D}_1(v)\cdot \gamma = D_1(v)$ and $\widetilde{in}(v)  \cdot \gamma =
\tin(v) $. 
A similar relation holds if one considers outputs and left actions.

\end{description}

We will denote by $\Gamma$ the quotient of the set $\widetilde{\Gamma}$ by the relation which is generated by all graph automorphisms. The set $\Gamma$ is the morphism set for the free prop $F(\X)$.  The decorations of graphs are obtained by formally composing elements of $X_1$. 

To elaborate, notice that we have an inclusion \[ X_1
\hookrightarrow \Gamma.\] Suppose that $x\in X_1$, $s(x) =
\lis{a_i}_{i=1}^n$, and $t(x) = \lis{b_k}_{k=1}^m$. The $n,m$ corolla
$C_{n,m}$ is the graph with one vertex which has $n$ incoming and $m$
outgoing edges. Choose an ordering for the incoming edges $e_1, \dots,
e_n$ and an ordering on the outgoing edges $e_{n+1}, \dots, e_{n+m}$.
Let $\mathfrak{g}(x)$ be the decoration on $C_{n,m}$ with $D_1(v) = x$,
$in  (v) = \lis{e_i}_{i=1}^n = \tin(C_{n,m})$, and $\tout(v) =
\lis{e_i}_{i=n+1}^{n+m} = \tout(C_{n,m})$. For this to be a decoration we must assign
\[ D_0(e_i) = \begin{cases}
a_i & i\in [1,n] \\
b_{i-n} & i\in [n+1, n+m]. 
\end{cases} \]

We remark that any other choice of order amounts to a graph automorphism
given by permuting the edges. 
The reader should also note that every decoration of the corolla $C_{n,m}$ is related to $\mathfrak{g}(x)$ for some $x$.

\subsection{Definition of prop structure on the collection \texorpdfstring{$\Gamma$}{\unichar{915}}}

We will now describe a prop $F(\X)$ whose set of morphisms is $\Gamma$ and whose color set $\col(F(\X))$ is  $X_0$. Let $\lis{a_i}_{i=1}^n$ and
$\lis{b_k}_{k=1}^m$ be lists of elements of $X_0$ and then define \[
F(\X)(\lis{a_i}_{i=1}^n; \lis{b_k}_{k=1}^m) \]  to be the set of all
equivalence classes of decorations such that \begin{align*} \M D_0 (\tin(G)) &= 
\lis{a_i}_{i=1}^n & \text{and} \\ \M D_0 (\tout(G)) &=
\lis{b_k}_{k=1}^m.\end{align*}

For each $c\in\col(F(\X)) = X_0$,  we will define the identity elements  $\id_c$, as decorations of the graph with
$E=\set{*}$ and $V=\emptyset$.

Vertical composition \begin{multline*} F(\X)(\lis{a_i}_{i=1}^n; \lis{b_k}_{k=1}^m) \times F(\X)(\lis{c_j}_{j=1}^p; \lis{a_i}_{i=1}^n) \\ \to F(\X)(\lis{c_j}_{j=1}^p ; \lis{b_k}_{k=1}^m) \end{multline*} is defined as follows. 
Let $\mathfrak{g^1}, \mathfrak{g^2}$ be a pair of decorations we wish to compose. 
We construct a graph $H$ with vertex set $V^1 \sqcup V^2$. If $\tin(G^1) = (e_1^1, \dots, e_n^1)$ and $\tout(G^2) = (e_1^2, \dots, e_n^2)$, then the edge set of $H$ is $(E^1 \sqcup E^2) /\sim$ where the $\sim$ is given by $e_i^1 \sim
e_i^2$. 
Before we had $t^2(e_i^2) = *$ and $s^1(e_i^1)=*$, but in the new graph $H$ we define $t^{1+2}(e_i^2) := t^1(e_i^1)$ and $s^{1+2}(e_i^1) := s^2(e_i^2)$.  
Let $G^{1+2}$ be the previously chosen representative of the isomorphism class of $H$, which will be the underlying graph of the composition $\mathfrak{g^1} \circ_v \mathfrak{g^2} =\mathfrak{g^{1+2}}$.

Now that we have defined the graph which is to be decorated, we can define the decoration. First, we set $D_1^{1+2} = D_1^1 \sqcup D_1^2$. The function $D_0^1 \sqcup D_0^2$ induces a function $D_0^{1+2}: E^{1+2} \to X_0$ since $\M D_0 (\tin(G^1))= \M D_0 (\tout(G^2))$. The ordering on the input and output edges of vertices are the same as those from $\mathfrak{g^1}$ and $\mathfrak{g^2}$: if, say, $v \in V_1$ and $\tin(v) = (a_1, \dots, a_\ell)$ in $\mathfrak{g^1}$, then in $\mathfrak{g^{1+2}}$ we have $\tin(v) = ([a_1], \dots, [a_\ell])$ where $[-]$ denotes the equivalence relation on $E^1 \sqcup E^2$. We define the order of $\tin(G^{1+2})$ to be the same as the order on $\tin(G^2)$ and the order of $\tout(G^{1+2})$ to be the same as the order on $\tout(G^1)$. 

We observe that an automorphism of $G^1$ or $G^2$ induces an automorphism of $G^{1+2}$, and thus we see that the definition of $\mathfrak{g^1} \circ_v \mathfrak{g^2}$ does not depend on
the choice of representatives in $\widetilde{\Gamma}$, but only the equivalence classes in $\Gamma$.

\begin{prop}
The vertical composition defined in the previous paragraph is associative.
\end{prop}
\begin{proof}
We wish to show that $(\mathfrak{g^1}\circ_v \mathfrak{g^2}) \circ_v
\mathfrak{g^3} = \mathfrak{g}^{1+2} \circ_v \mathfrak{g^3} =
\mathfrak{g^{(1+2) +3}}$ is equal to $\mathfrak{g^1}\circ_v
(\mathfrak{g^2} \circ_v \mathfrak{g^3}) = \mathfrak{g}^{1} \circ_v
\mathfrak{g^{2+3}} = \mathfrak{g^{1+(2 +3)}}$. It is clear that
$V^{(1+2) + 3} = V^{1+(2+3)}$ since disjoint union is associative. When
forming $E^{(1+2) + 3}$ we first identify $\tout(G^2)$ and $\tin(G^1)$ and
then identify $\tout(G^3)$ with $\tin(G^{1+2}) = \tin(G^2)$. This is the same
as first identifying $\tin(G^2)$ with $\tout(G^3)$ and then identifying
$\tout(G^{2+3}) = \tout(G^2)$ with $\tin(G^1),$ so $E^{(1+2)+3} =
E^{1+(2+3)}$. At this point we  see that $G^{(1+2)+3}$ and $G^{1+(2+3)}$
as defined above have the same source and target maps. 

The decorations $\mathfrak{g^{(1+2) +3}}$ and $\mathfrak{g^{1+(2 +3)}}$
are identical given the way they are induced from $\mathfrak{g^1}$,
$\mathfrak{g^2}$, and $\mathfrak{g^3}$. \end{proof}

The horizontal composition  $\mathfrak{g^1} \circ_h \mathfrak{g^2}
=\mathfrak{g^{1+2}}$ of two decorations $\mathfrak{g^1}$ and
$\mathfrak{g^2}$ is given by disjoint union of the underlying graphs
with decoration given by disjoint union. The only thing we must declare
is that $\tin(G^{1+2})$ is ordered so that $\tin(G^1) < \tin(G^2)$ and
$\tout(G^{1+2})$ is ordered so that $\tout(G^1) < \tout(G^2)$. This is clearly
associative since (ordered) disjoint union is associative.

The symmetric action is given by the action on (ordered) input and
output edges of the graph. If $\mathfrak{g}$ has input edges $\tin(G) \in
\M E$ then $\sigma^*\mathfrak{g} $ has exactly the same structure as
$\mathfrak{g}$ except that we give the input vertices of $G$ the order
$\tin(G)\cdot \sigma$. Similarly, $\tau _*\mathfrak{g}$ just has the
modified order $\tau \cdot \tout(G)$ on output vertices.

\begin{lem} The vertical composition is compatible with the symmetric
group actions in the sense that \begin{align} \mathfrak{g^1} \circ_v
(\sigma_* \mathfrak{g^2}) &= (\sigma^* \mathfrak{g^1}) \circ_v
\mathfrak{g^2} \label{E:inbetween} \\ \sigma^*(\mathfrak{g^1}\circ_v
\mathfrak{g^2}) &= \mathfrak{g^1} \circ_v (\sigma^* \mathfrak{g^2})
\label{E:axiomvertupper} \\ \tau_*(\mathfrak{g^1}\circ_v \mathfrak{g^2})
&= (\tau_* \mathfrak{g^1})\circ_v \mathfrak{g^2}.
\label{E:axiomvertlower} \end{align} \end{lem}
\begin{proof} If $\mathfrak{g}$ is a decoration of $G$, then we will
write $\tin(\mathfrak{g})$ for $\tin(G)$ and $\tout(\mathfrak{g})$ for
$\tout(G)$ since we will be dealing with so many orders in this proof.

For \eqref{E:inbetween} it is enough to check that the underlying graphs
are the same. But we form graph for the left-hand decoration by
identifying the ordered sets $\tin(\mathfrak{g^1}) = (e_1^1, \dots,
e_n^1)$ and \[ \tout(\mathfrak{\sigma_* \mathfrak{g^2}}) = \sigma\cdot
\tout(\mathfrak{\mathfrak{g^2}}) = (e_{\sigma^\inv(1)}^2, \dots,
e_{\sigma^\inv(n)}^2) \] and we form the graph for the right-hand
decoration by identifying the ordered sets
$\tout(\mathfrak{\mathfrak{g^2}}) = (e_1^2, \dots, e_n^2)$ and \[
\tin(\mathfrak{\sigma^* \mathfrak{g^1}}) = 
\tin(\mathfrak{\mathfrak{g^1}})\cdot \sigma = (e_{\sigma(1)}^1, \dots,
e_{\sigma(n)}^1). \] In the first case we are identifying $e_i^1 \sim
e_{\sigma^\inv(i)}^2$ while in the second we are identifying $e_i^2 \sim
e_{\sigma(i)}^1$; since this is the same identification we see that the underlying
graph of both compositions is the same.

The only part of a decoration that $\sigma^*$ changes is the order on
$\tin(G)$, so \[ \tin(\sigma^*(\mathfrak{g^1}\circ_v \mathfrak{g^2})) =
\tin(\mathfrak{g^1}\circ_v \mathfrak{g^2}) \cdot \sigma =
\tin(\mathfrak{g^2})\cdot \sigma = \tin(\sigma^* \mathfrak{g^2}) = in
(\mathfrak{g^1} \circ_v (\sigma^* \mathfrak{g^2})) \] gives
\eqref{E:axiomvertupper}. Equation \eqref{E:axiomvertlower} follows
similarly. \end{proof}

\begin{prop}
The definitions above make $F(\X)$ into a prop.
\end{prop}
\begin{proof} It is immediate from construction of $\circ_v$ that
\eqref{E:identities} holds.

It remains to show that the interchange of the horizontal and vertical
compositions holds \[  (\mathfrak{g^1} \circ_v \mathfrak{g^2}) \circ_h
(\mathfrak{g^3} \circ_v \mathfrak{g^4}) = (\mathfrak{g^1}\circ_h
\mathfrak{g^3}) \circ_v (\mathfrak{g^2}\circ_h \mathfrak{g^4}). \] We
examine the underlying graphs on each side. They have the same vertex
set $V_1 \sqcup V_2 \sqcup V_3 \sqcup V_4$. The edge set on each side is
$E_1 \sqcup E_2 \sqcup E_3 \sqcup E_4$ with some of the edges
identified. On the left, we identify the ordered sets $\tin(G^1) \sim
\tout(G^2)$ as well as $\tin(G^3) \sim \tout(G^4)$. On the right we identify
$( \tin(G^1), \tin(G^3)) \sim (\tout(G^2), \tout(G^4))$. These identifications
are the same, and we see that both sides have the same underlying graph.
As the structure of the compositions is just induced from that on the
individual decorations, these decorations are the same as well.

We showed that the vertical composition is compatible with the symmetric
group actions in the previous lemma. The compatibility of horizontal
composition with symmetric group actions as in \eqref{E:horizcompat} and
\eqref{E:horizswap} is easy to see from the definition of $\circ_h$.
Finally, the interchange rule $\sigma^* \tau_* = \tau_* \sigma^*$ is
obvious since $\sigma^*$ only modifies the order on $\tin(G)$ and $\tau_*$
only modifies the order on $\tout(G)$.

\end{proof}

\begin{prop} The assignment $F: \Mega \to \Prop$ is a functor.
\end{prop}
\begin{proof} Suppose that $f=(f_1, f_0): \X \to \Y$ is a morphism of
megagraphs, where $f_i: X_i\to Y_i$; we wish to describe $F(f)$. Let $\mathfrak{g}$ be in
$F(\X)$, i.e. $\mathfrak{g}$ is a decoration of a graph $G$ by $\X$. We
then have a decoration of $G$ by $\Y$, which we denote $\mathfrak{g^f},$
with $D_0^f = f_0\circ D_0$ and $D_1^f = f_1\circ D_1$, and the order
structures are the same. We define $F(f)\mathfrak{g}:=\mathfrak{g^f}$. Since
composition of set functions is associative, $F(f\circ g) = F(f) \circ
F(g)$. \end{proof}

\subsection{The functors \texorpdfstring{$F$ and $U$}{F and U} are an adjoint pair}

We now turn to adjointness. We will show that \[ \Hom_{\Prop}(F(\X), \T)
= \Hom_{\Mega}(\X, U(\T)), \] where $U$ is the forgetful functor. As
mentioned above, we have inclusions \begin{align*} X_0 &\hookrightarrow
\Gamma \\ X_1 &\hookrightarrow \Gamma \end{align*} so we must 
show that given $f: \X \to U(\T)$ there \emph{exists} a \emph{unique}
map of props $K: F(\X) \to \T$ such that $K|_{X_0}=f_0$ and
$K|_{X_1}=f_1$. Existence and uniqueness will be shown at the same time
through an inductive process. Since elements of $\Gamma$ have an
underlying graph, we can define a filtration $\Gamma_0 \ci \Gamma_1 \ci
\dots$ based on the order of the underlying graph; we define $K$ as a
limit of partially defined functors $K_p: \Gamma_p \to \T$.  Since we must choose an order on inputs and outputs in order to define $\mathfrak{g}(x)$,  we fix this choice for all corollas $G$ in the process of defining $K_p$. In what follows, we will show that each $K_p$ can then be defined in exactly one way.

Here is what we require from $K_p: \Gamma_p \to \T$:
\begin{itemize}
	\item Compatibility with $f$:
		\begin{itemize}
			\item If $c\in X_0 = \col(F(\X))$, then $K_p (\id_c) = \id_{f_0(c)}$.
			\item If $x\in X_1$ then $K_p(\mathfrak{g}(x)) = f_1(x)$.
		\end{itemize}
	\item The $K_p$ constitute a filtration: $K_p|_{\Gamma_{p-1}} = K_{p-1}$.
	\item Partial functoriality:
		\begin{itemize}
			\item If $\mathfrak{g} \in \Gamma_p$ and $\mathfrak{g}=\mathfrak{g^1} \circ_h \mathfrak{g^2}$, then $K_p(\mathfrak{g}) = K_p(\mathfrak{g^1}) \circ_h K_p(\mathfrak{g^2})$.
			\item If $\mathfrak{g} \in \Gamma_p$ and $\mathfrak{g}=\mathfrak{g^1} \circ_v \mathfrak{g^2}$, then $K_p(\mathfrak{g}) = K_p(\mathfrak{g^1}) \circ_v K_p(\mathfrak{g^2})$.
			\item If $\mathfrak{g} \in \Gamma_p$ then $K_p(\sigma^*\tau_* \mathfrak{g}) = \sigma^*\tau_* K_p(\mathfrak{g})$.
		\end{itemize}
\end{itemize}

The graphs with zero vertices are just a collection of non-incident
edges; we define \[ K_0 (\sigma^* \tau_* (\id_{c_1} \circ_h \dots
\circ_h \id_{c_k})) = \sigma^*\tau_* (\id_{f_0(c_1)} \circ_h \dots \circ_h
\id_{f_0(c_k)}). \] We define $K_1|_{\Gamma_0} = K_0$, and
\begin{multline*} K_1 (\sigma^* \tau_* (\id_{c_1} \circ_h \dots \circ_h
\id_{c_k} \circ_h \mathfrak{g}(x) \circ_h \id_{c_{k+1}} \circ_h \dots
\circ_h \id_{c_\ell})) \\= \sigma^*\tau_* (\id_{f_0(c_1)} \circ_h \dots
\circ_h \id_{f_0(c_k)} \circ_h f_1(x) \circ_h \id_{f_0(c_{k+1})} \circ_h
\dots \circ_h \id_{f_0(c_\ell)}), \end{multline*} which covers all
decorations on order 1 graphs. Note that $K_0$ and $K_1$ are
well-defined and satisfy the above conditions.

We now build $K_p$ using $K_{p-1}$. For this, we need a suitable
collection of subgraphs. By a \emph{subgraph} of $G=(E,V)$ we mean a
pair of subsets $E^0 \ci E$ and $V^0 \ci V$. These determine a graph
with source and target maps induced by those from $G$, i.e. they are
defined by \[ s^0,t^0: E^0 \hookrightarrow E
\overset{s,t}\longrightarrow V_+ \twoheadrightarrow V^0_+ \]so that
$s^0(e) = s(e)$ and $t^0(e) = t(e)$ whenever possible (the second arrow
is given by $(V\setminus V^0) \mapsto *$). An \emph{admissible subgraph}
is a subgraph satisfying the condition that if $v\in V^0$ then any edge
incident to $v$ is in $E^0$.

Notice that if $G^0$ is an admissible subgraph of $G$, a decoration
$\mathfrak{g}$ on $G$ nearly induces a decoration on $G^0$. The only
thing that is missing is an order on the input and output edges.

\begin{defn} A \emph{decomposition} of a graph $G$ is a collection of
admissible subgraphs $G^1, \dots, G^n$ so that $V^i \cap V^j =
\emptyset$ when $i\neq j$, $V =  \bigcup V^i$, and $E= \bigcup E^i$. A
\emph{proper} decomposition is one in which each $V^i$ is nonempty.
\end{defn}

Notice that the intersection of two decompositions is again a
decomposition, where by intersection of $G^1, \dots, G^n$ and
$\bar{G}^1, \dots, \bar{G}^m$ we mean \[ G^1 \cap \bar{G}^1, G^1 \cap
\bar{G}^2, \dots, G^n \cap \bar{G}^{m-1}, G^n \cap \bar{G}^m. \]

We now isolate a particularly interesting type of decomposition. A
\emph{vertical decomposition} of $G$ is a decomposition $G^1, \dots,
G^n$ so that \begin{align*} \tout(G^1) &= \tout(G)\\ \tin(G^n)&=\tin(G) \\
\text{and } \tout(G^i) &= \tin(G^{i-1}) \text{ for } 2 \leq i \leq n.
\end{align*} If $\mathfrak{g}$ is a decoration on $G$, then a
\emph{vertical decomposition of $\mathfrak{g}$} is a vertical
decomposition of $G$ together with a choice of orders on $\tout(G^i)$ and $\tin(G^i)$ for
$i\in [1,n]$ so that the above equalities hold \emph{as ordered sets.}
The data of a vertical decomposition thus gives decorations
$\mathfrak{g^i}$ on $G^i$, and, moreover, $\mathfrak{g} = \mathfrak{g^1}
\circ_v \cdots \circ_v \mathfrak{g^n}$.

\begin{lem} Suppose that $G^0$ is an admissible subgraph of $G$ and
$G^1, G^2$ is a vertical decomposition of $G$. Define subgraphs $G^{01}$
and $G^{02}$ by \begin{align*} V^{01} &= V^0 \cap V^1 & V^{02} &= V^0
\cap V^2 \\ E^{01} &= (E^0 \cap E^1) \cup (\tout(G^0)) & E^{02} &= (E^0
\cap E^2) \cup (\tin(G^0)). \end{align*} Then $G^{01}, G^{02}$ is a
vertical decomposition of $G^0$. \end{lem}
\begin{proof} To show that $G^{0i}$ is admissible, suppose that $v\in
V^{0i} = V^0 \cap V^i$. If $e$ is incident to $v$ in $G^0$, then $e$ is
incident to $v$ in $G$, so by admissibility of $G^i$ and $G^0$ we have
$v\in E^0 \cap E^i \ci E^{0i}$.

We now show that $G^{01}$, $G^{02}$ constitutes a decomposition of
$G^0$. A vertex $v\in V^0$ must either be in $V^1$ or $V^2$ since $G^1$,
$G^2$ is a decomposition of $G$, so $v\in V^0\cap V^1 = V^{01}$ or
$V^0\cap V^2=V^{02}$. The same argument holds for edges. We also need to
check disjointness of $V^{01}$ and $V^{02}$, but $V^{01} \cap V^{02} =
V^0 \cap V^1 \cap V^2 = V^0 \cap \emptyset$.

Showing that this is a \emph{vertical} decomposition is a little bit
more work. Remember that we are trying to show that $\tout(G^{01}) =
\tout(G^0)$, and $\tin(G^{01}) = \tout(G^{02})$, and $\tin(G^{02}) = \tin(G^0)$. We
will just show that $\tout(G^{01}) = \tout(G^0)$, and $\tout(G^{02}) \ci
\tin(G^{01})$ since the proofs of the other equality other inclusion are
dual.

The inclusion $\tout(G^{0}) \ci \tout(G^{01})$ is part of the definition of
$G^{01}$. In order for $e\in \tout(G^{01})$, we must have $t^0(e) = *$ or
$t^1(e)=*$. If $t^1(e) = *$, then $e\in \tout(G^1) = \tout(G)$, so $t(e) =
*$. Thus $t^0(e) = *$ since $G^0$ is a subgraph of $G$. So $e\in
\tout(G^{01})$ implies that $t^0(e) = *$, so $e\in \tout(G^0)$.

We now wish to show that $\tout(G^{02}) \ci \tin(G^{01})$. If $e\in
\tout(G^{02})$, then $t^{02}(e)= *$ which means either $t^0(e) = *$ or
$t^2(e)=*$. Let us first check that such an $e$ is actually in $E^{01}$.
If $t^0(e) = *$, then $e\in \tout(G^0) \ci E^{01}$. If $t^2(e) = *$, then
$e\in  \tout(G^2) = \tin(G^1) \ci E^1$. Therefore $t^2(e) = *$ implies
$e\in E^0 \cap E^1 \ci E^{01}$.

Now that we know that $\tout(G^{02}) \ci E^{01}$, we must show that
$s^{01}(e) = *$ for $e\in \tout(G^{02})$. Towards this end, we examine
\begin{align*}
\tout(G^{02}) &= [E^0 \cap \tout(G^2)] \cup [\tin(G^0) \cap \tout(G^0)] \cup [E^2 \cap \tout(G^{0})] \\
&= [E^0 \cap \tin(G^1)] \cup [\tin(G^0) \cap \tout(G^0)] \cup [E^2 \cap \tout(G^{0})],
\end{align*}
which we now prove by showing the top equality; the bottom follows since
$G^1, G^2$ is a vertical decomposition. The inclusion from right to left
is easy, since any element $e$ in the right hand set has $t^0(e) = *$ or
$t^2(e) = *$, so $t^{02}(e) = *$. If $e\in \tout(G^{02})$ and $t^2(e)=*$
then $e\in \tout(G^2)$ so we are in the first set on the right. If $t^0(e)
= *$ then $e\in \tout(G^0)$. Since $e\in E^{02}$ we either have $e\in E^2$
or $e\in \tin(G^0)$, so $e$ is contained in one of the two rightmost sets.
Thus the left hand side is contained in the right hand side, and we have
shown equality.

It is clear that $E^0 \cap \tin(G^1)$ and $\tin(G^0) \cap \tout(G^0)$ are
contained in $\tin(G^{01})$. Our remaining work then is to show that $E^2
\cap \tout(G^0) \ci \tin(G^{01})$. 

We first make an observation. If $e\in E^2$ then $s(e) \notin V^1$. If
it were, then $s^2(e) = *$ by disjointness of $V^1$ and $V^2$, so $s(e)
= *$ since $\tin(G^2) = \tin(G)$. But we cannot have both $s(e) = *$ and
$s(e) \in V^1$.

Now consider $e\in E^2 \cap \tout(G^0)$. We either have $s(e) \in V^2$ or
$s(e) = *$. In both cases we have $s^1(e) = *$, so $e\in \tin(G^{01})$.
Thus we have shown $\tout(G^{02}) \ci \tin(G^{01})$, which completes the
proof.
\end{proof}

Let us now consider two vertical decompositions $G^1, G^2$ and $G^3,
G^4$ of the same graph $G$. Using the same notation from the previous
proposition, we have decompositions $G^{13}, G^{14}$ and $G^{23},
G^{24}$ of $G^1$ and $G^2$ respectively, and  decompositions $G^{31},
G^{32}$ and $G^{41}, G^{42}$ of $G^3$ and $G^4$. Notice that $V^{ij} =
V^{ji}$ for $i=1,2$ and $j=3,4$. 

As for the edge sets, we have
\begin{align*}
E^{13} &= (E^1 \cap E^3) \cup \tout(G^1) \\ 
	&= (E^1 \cap E^3) \cup \tout(G) \\
	&= (E^3 \cap E^1) \cup \tout(G^3) 
	= E^{31}
\end{align*}
and similarly $E^{24} = E^{42}$.
Furthermore, 
\begin{equation}\label{eq:edgesequal}
E^{14} \cup E^{23} = (E^1 \cap E^4) \cup (E^2 \cap E^3) = E^{41} \cup E^{32};
\end{equation}
to see this, it is enough to check that $E^{14}, E^{23}, E^{41},$ and $E^{32}$ are subsets of $(E^1 \cap E^4) \cup (E^2 \cap E^3)$.
To show, for example, that $(E^1\cap E^4) \cup \tin(G^1) =  E^{14} \subset (E^1 \cap E^4) \cup (E^2 \cap E^3)$, take an edge $e\in \tin(G^1) = \tout(G^2) \subset E^1 \cap E^2$. Since $e$ is in $G$ and $G^3, G^4$ is a decomposition, either $e\in E^3$ or $e\in E^4$, which implies that 
\[ e\in (E^1 \cap E^2) \cap E^3 \subset E^2 \cap E^3 \text{ or } e\in (E^1 \cap E^2) \cap E^4 \subset (E^1 \cap E^4), \]
hence $e\in (E^1 \cap E^4) \cup (E^2 \cap E^3)$. The proofs that $E^{23}, E^{41}, E^{32} \subset (E^1 \cap E^4) \cup (E^2 \cap E^3)$ are similar, so \eqref{eq:edgesequal} holds and 
we thus have identical vertical compositions
\begin{equation}\label{newdecompositions}
\begin{gathered}
G^{13}, (G^{14} \cup G^{23}), G^{24} \\
G^{31}, (G^{32} \cup G^{41}), G^{42}.
\end{gathered}
\end{equation}
Note that there are no edges of $G$ between vertices in $V^{14}$ and $V^{23}$, hence, if $G$ is \emph{connected} then $G^{14} \cup G^{23}$ has strictly fewer vertices than $G$. 

\begin{prop}\label{P:verticalsplitting} Suppose that $G$ is a connected
graph of order $p$, $\mathfrak{g}$ is a decoration on $G$, and $K_{p-1}$ is defined
and satisfies the required properties. If $\mathfrak{g^1} \circ_v
\mathfrak{g^2}$ and $\mathfrak{g^3} \circ_v \mathfrak{g^4}$ are two
proper vertical decompositions of $\mathfrak{g}$, then \[
K_{p-1}(\mathfrak{g^1}) \circ_v K_{p-1}(\mathfrak{g^2}) = K_{p-1} (
\mathfrak{g^3}) \circ_v K_{p-1}(\mathfrak{g^4}). \] \end{prop}
\begin{defn} If $G$ is connected, we define  \[ K_p(\mathfrak{g}) :=
K_{p-1}(\mathfrak{g^1}) \circ_v K_{p-1}(\mathfrak{g^2})\] for any proper
vertical decomposition $\mathfrak{g^1}, \mathfrak{g^2}$ of
$\mathfrak{g}$. \end{defn}

\begin{rem}
In the setting of the previous definition, we have
\begin{multline*} K_p(\sigma^* \tau_* \mathfrak{g})
= K_{p-1}(\sigma^* \mathfrak{g^1}) \circ_v K_{p-1}(\tau_*
\mathfrak{g^2})  = \sigma^* K_{p-1}( \mathfrak{g^1}) \circ_v \tau_*
K_{p-1}( \mathfrak{g^2}) \\= \sigma^* \tau_* \left( K_{p-1}(
\mathfrak{g^1}) \circ_v  K_{p-1}( \mathfrak{g^2})\right) = \sigma^*
\tau_* K_p (\mathfrak{g}),\end{multline*} 
so we see that compatibility
with symmetric group actions follows from the same property on $K_{p-1}$.
\end{rem}

\begin{proof}[Proof of Proposition~\ref{P:verticalsplitting}]
We will write $G^\dagger$ for $G^{14} \cup G^{23}$. 
Choose an order on
$\tin(G^{13})$ and on $\tout(G^{42})$, and use the appropriate orders on
everywhere else:
\begin{align*}
\tout(G^{13}) &= \tout(G) & \tin(G^{42}) &= \tin(G) \\
\tin(G^{14}) &= \tin(G^1) & = \tout(G^{23}) &= \tout(G^2) \\
\tin(G^{32}) &= \tin(G^3) & = \tout(G^{41}) &= \tout(G^4) \\
\tin(G^\dagger) &= \tout(G^{42}) & \tout(G^\dagger) &= \tin(G^{13}).
\end{align*}
Using \eqref{newdecompositions}, we now have an additional vertical decomposition 
\[ \mathfrak{g} =
\mathfrak{g^{13}} \circ_v \mathfrak{g^\dagger} \circ_v \mathfrak{g^{24}}
\] 
of $\mathfrak{g}$, as well decompositions
\begin{align*}
\mathfrak{g^1} &= \mathfrak{g^{13}} \circ_v \mathfrak{g^{14}} &
\mathfrak{g^2} &= \mathfrak{g^{23}} \circ_v \mathfrak{g^{24}} \\
\mathfrak{g^3} &= \mathfrak{g^{31}} \circ_v \mathfrak{g^{32}} &
\mathfrak{g^4} &= \mathfrak{g^{41}} \circ_v \mathfrak{g^{42}} \\
\mathfrak{g^\dagger} &= \mathfrak{g^{14}} \circ_v \mathfrak{g^{23}} 
= \mathfrak{g^{32}} \circ_v \mathfrak{g^{41}}.
\end{align*}
The decorations $\mathfrak{g^\dagger}$ and each $\mathfrak{g^{i}}$ have fewer than $p$ vertices using the fact that $G$ is connected and properness.
From now on we will only write down decorations with fewer than $p$ vertices, so we can safely write $K$ instead of $K_{p-1}$.

\begin{align*}
K(\mathfrak{g^{13}}) \circ_v K(\mathfrak{g^\dagger}) 
&= K(\mathfrak{g^{13}}) \circ_v K(\mathfrak{g^{14}})\circ_v K(\mathfrak{g^{23}}) \\
&= K(\mathfrak{g^{1}})\circ_v K(\mathfrak{g^{23}}) \\
K(\mathfrak{g^{31}}) \circ_v K(\mathfrak{g^\dagger}) 
&= K(\mathfrak{g^{31}}) \circ_v K(\mathfrak{g^{32}})\circ_v K(\mathfrak{g^{41}}) \\
&= K(\mathfrak{g^{3}})\circ_v K(\mathfrak{g^{41}}) \\
\end{align*}
These are equal, so we compose with $K(\mathfrak{g^{24}})$ and find
\begin{gather*}
K(\mathfrak{g^1}) \circ_v K(\mathfrak{g^{23}}) \circ_v K(\mathfrak{g^{24}}) = 
K(\mathfrak{g^1}) \circ_v K(\mathfrak{g^2})
\\
K(\mathfrak{g^{3}})\circ_v K(\mathfrak{g^{41}}) \circ_v K(\mathfrak{g^{42}}) =
K(\mathfrak{g^3}) \circ_v K(\mathfrak{g^4})
\end{gather*}
are equal as well.
\end{proof}

We now move on to arbitrary decompositions. At the moment, we have only
defined $K_p$ for \emph{connected} decompositions on $p$ vertices and for
arbitrary decompositions on \emph{fewer} than $p$ vertices.

Let $\mathfrak{g^1}, \dots, \mathfrak{g^k}$ be decorations whose
underlying graphs are connected, so that $\sum |V^i| = p$. We define \[
K_p(\mathfrak{g^1} \circ_h \cdots \circ_h \mathfrak{g^k} ) =
K_p(\mathfrak{g^1}) \circ_h \cdots \circ_h K_p(\mathfrak{g^k} ), \]
which is well-defined since $K_p$ is well-defined whenever the
underlying graph is connected. In this same situation we assign \[
K_p(\sigma^* \tau_* ( \mathfrak{g^1} \circ_h \cdots \circ_h
\mathfrak{g^k} )) = \sigma^* \tau_* K_p(\mathfrak{g^1} \circ_h \cdots
\circ_h \mathfrak{g^k} ).\] Let us see that this is well-defined.
Suppose  that \[ \sigma^* \tau_* ( \mathfrak{g^1} \circ_h \cdots \circ_h
\mathfrak{g^k} ) = \bar\sigma^* \bar\tau_* ( \mathfrak{\bar g^1} \circ_h
\cdots \circ_h \mathfrak{\bar g^k} ), \] and move the symmetric group
actions to the other side. We then see
\begin{align*}
\mathfrak{g^1} \circ_h \cdots \circ_h \mathfrak{g^k} &= (\bar\sigma \sigma^\inv)^* (\tau^\inv \bar\tau)_*( \mathfrak{\bar g^1} \circ_h \cdots \circ_h \mathfrak{\bar g^k} ) \\
&= \sigma_1^* \tau^1_* \mathfrak{\bar g^{a_1}} \circ_h \cdots \circ_h \sigma_k^* \tau^k_* \mathfrak{\bar g^{a_k}},
\end{align*}
whence
\begin{align*} K_p(\mathfrak{g^1} \circ_h \cdots \circ_h \mathfrak{g^k} ) 
&= K_p( \sigma_1^* \tau^1_* \mathfrak{\bar g^{a_1}} \circ_h \cdots \circ_h \sigma_k^* \tau^k_* \mathfrak{\bar g^{a_k}} ) \\
&= K_p( \sigma_1^* \tau^1_* \mathfrak{\bar g^{a_1}}) \circ_h \cdots \circ_h K_p(\sigma_k^* \tau^k_* \mathfrak{\bar g^{a_k}} ) \\
&= \sigma_1^* \tau^1_*  K_p( \mathfrak{\bar g^{a_1}}) \circ_h \cdots \circ_h \sigma_k^* \tau^k_*K_p( \mathfrak{\bar g^{a_k}} ) \\
&= (\bar\sigma \sigma^\inv)^* (\tau^\inv \bar\tau)_*( K_p( \mathfrak{\bar g^1}) \circ_h \cdots \circ_h K_p(\mathfrak{\bar g^k})).
\end{align*}
Applying $\sigma^*\tau_*$ to both sides gives 
\[ \sigma^* \tau_* K_p (\mathfrak{g^1} \circ_h \cdots \circ_h \mathfrak{g^k} ) = \bar\sigma^* \bar\tau_* K_p( \mathfrak{\bar g^1} \circ_h \cdots \circ_h \mathfrak{\bar g^k} ).\]

\begin{proof}[Summary of Proof of Theorem~\ref{T:freeprop}]
We wished to show \[ \Hom_{\Prop}(F(\X), \T) = \Hom_{\Mega}(\X, U(\T)),
\] which amounted to showing that $f:\X \to U(\T)$ induces a unique prop
map $K: F(\X) \to \T$ with $K|_{X_0} = f_0$ and $K|_{X_1} = f_1$. We
filtered the set of all morphisms, $\Gamma$, by the number of vertices
of the underlying graph of the decoration. We then built the prop map
$K$ inductively, with partially defined prop maps $K_p: \Gamma_p \to
\T$.

Since $K_p$ needed to agree with $K_{p-1}$ whenever possible, we could
\emph{define} $K_p(\mathfrak{g})$ to be $K_{p-1}(\mathfrak{g^1}) \circ_v
K_{p-1}(\mathfrak{g^2})$ or $K_{p-1}(\mathfrak{g^1}) \circ_h
K_{p-1}(\mathfrak{g^2})$ wherever this was possible and made sense. This
entire section was devoted to showing that such an assignment was
well-defined. The properties essentially came for free. \end{proof}

\bibliographystyle{amsplain} \bibliography{infinity}

\end{document}